\documentclass[usenames,dvipsnames]{article}

\usepackage{amsmath,amsfonts,amsthm}
\usepackage[justification=justified, singlelinecheck=false]{caption}
\usepackage{graphicx}
\usepackage{a4wide}
\usepackage{bbm}
\usepackage{lipsum}
\usepackage{epstopdf}
\usepackage{algorithmic}
\usepackage{array}  
\usepackage{tikz}
\usepackage[scaled=.90]{helvet}
\usepackage{courier}
\usepackage{ae}
\usepackage[T1]{fontenc}
\usepackage{subfigure}
\usepackage{color}
\usepackage{xcolor}
\usepackage[T1]{fontenc}
\usepackage{here} 
\usepackage{cleveref}

\newtheorem{remark}{Remark}[]

\newtheorem{corollary}{Corollary}[]
\newtheorem{proposition}{Proposition}[]

\theoremstyle{definition}

{\begin{proof}[Beweis]}
	{\end{proof}}

\newtheorem{notation}{Notation}

\usepackage{xpatch}

\usepackage[symbol]{footmisc}

\begin{document}
		
\title{Qualitative properties of numerical methods for the inhomogeneous geometric Brownian motion}
\author{Irene Tubikanec\footnotemark[1]\thanks{Institute for Stochastics, Johannes Kepler University Linz (irene.tubikanec@jku.at,evelyn.buckwar@jku.at)}, Massimiliano Tamborrino\footnotemark[2]\thanks{Department of Statistics, University of Warwick (massimiliano.tamborrino@warwick.ac.uk)}, Petr Lansky\footnotemark[3]\thanks{Institute of Physiology, Czech Academy of Sciences (lansky@biomed.cas.cz)}, Evelyn Buckwar\footnotemark[1] \footnotemark[4]\thanks{Centre for Mathematical Sciences, Lund University}\\}

\date{}
\maketitle	
	
\thispagestyle{empty}

\section*{Abstract}

We provide a comparative analysis of qualitative features of different numerical methods for the inhomogeneous geometric Brownian motion (IGBM). The conditional and asymptotic mean and variance of the IGBM are known and the process can be characterised according to Feller's boundary classification. We compare the frequently used Euler-Maruyama and Milstein methods, two Lie-Trotter and two Strang splitting schemes and two methods based on the ordinary differential equation (ODE) approach, namely the classical Wong-Zakai approximation and the recently proposed log-ODE scheme. First, we prove that, in contrast to the Euler-Maruyama and Milstein schemes, the splitting and ODE schemes preserve the boundary properties of the process, independently of the choice of the time discretisation step. Second, we derive closed-form expressions for the conditional and asymptotic means and variances of all considered schemes and analyse the resulting biases. While the Euler-Maruyama and Milstein schemes are the only methods which may have an asymptotically unbiased mean, the splitting and ODE schemes perform better in terms of variance preservation. The Strang schemes outperform the Lie-Trotter splittings, and the log-ODE scheme the classical ODE method. The mean and variance biases of the log-ODE scheme are very small for many relevant parameter settings. However, in some situations the two derived Strang splittings may be a better alternative, one of them requiring considerably less computational effort than the log-ODE method. The proposed analysis may be carried out in a similar fashion on other numerical methods and stochastic differential equations with comparable~features.

\vspace{-0.2cm}
\subsubsection*{Keywords} Geometric Brownian motion, Inhomogeneous drift, Feller's boundary classification, Numerical splitting schemes, Boundary preservation, Moment preservation

\vspace{-0.2cm}
\subsubsection*{AMS subject classifications} 60H10, 60H35, 65C20, 65C30

\vspace{-0.2cm}
\subsubsection*{Acknowledgements} The authors would like to thank James Foster for his interest in this manuscript and for his helpful input regarding the conditional moments of the log-ODE method. This work was supported by the Austrian Exchange Service (OeAD), bilateral project CZ 19/2019 and by the Austrian Science Fund (FWF), W1214-N15, project DK 14. 

\vspace{0.5cm}

\section{Introduction}
\vspace{-0.2cm}

The inhomogeneous geometric Brownian motion (IGBM), described by the It\^o stochastic differential equation (SDE)
\vspace{-0.2cm}
\begin{equation*}
dY(t)=\left(-\frac{1}{\tau}Y(t)+\mu\right) dt+\sigma Y(t) dW(t), \quad t \geq 0, \quad Y(0)=Y_0, \vspace{-0.1cm}
\end{equation*}
is frequently applied in mathematical and computational finance, neuroscience and other fields. In particular, it is often used to describe price fluctuations in finance \cite{Capriotti2018,Zhao2009} or changes in the neuronal membrane voltage in neuroscience \cite{Donofrio2018}. This process is also known as geometric Brownian motion (GBM) with affine drift~\cite{Linetsky2004}, geometric Ornstein-Uhlenbeck (OU) process \cite{Insley2002} or mean reverting GBM \cite{Sarkar2003} in real option theory, as Brennan-Schwarz model \cite{Brennan1979,Chan1992} in the interest rate literature, as GARCH model \cite{Barone2005,Minqiang2018} in stochastic volatility and energy markets, as Lognormal diffusion with exogenous factors \cite{Guiterrez1997} in growth analysis and forecasting or as reciprocal gamma diffusion in \cite{Leonenko2012}. The IGBM is a multiplicative noise process, characterised by an inhomogeneous drift term, defined through $\mu \in \mathbb{R}$, and can be seen as an illustrative equation for this class of SDEs. In particular, it 
is a member of the Pearson diffusion class \cite{Sorensen2007}. Differently from other well-known Pearson diffusions, such as the OU process \cite{Arnold1974,Lansky1995} and the square-root process \cite{Cox1985,Ditlevsen2006,Feller1951,Lansky1995}, the transition density of the IGBM does not have a practical closed-form expression \cite{Zhao2009} and an exact simulation method is not available. Hence, we need to rely on numerical methods that accurately reproduce the features of the process, making its analysis and investigation via simulations possible and~reliable.

A large part of the area of (stochastic) numerical analysis is devoted to convergence of numerical methods in a suitable sense. These are limit results for the time discretisation step going to zero over a finite interval and, of course, numerical methods which do not converge should not be used. Nevertheless, in practice, a strictly positive time step is required. In consequence, the numerical method can be viewed as the solution of a discrete dynamical system, which may or may not have the same properties and behaviour as the solution of the original problem \cite{Hairer2006}. In the worst case, although the method converges, the discretisation step may alter the essential properties of the model, making the numerical method practically useless or very inefficient. 

The purpose of this article is to analyse and compare different numerical methods regarding their ability to preserve qualitative features of the IGBM for a fixed time discretisation step. In particular, we focus on methods based on the splitting and ordinary differential equation (ODE) approaches, and on their comparison with the commonly used Euler-Maruyama and Milstein~schemes. 

The idea behind the splitting approach is to split the equation of interest into explicitly solvable subequations, and to apply a proper composition of the resulting exact solutions. A standard procedure is the Lie-Trotter composition \cite{Trotter1959}, and a less commonly analysed method is the Strang approach \cite{Strang1968}. We refer to \cite{Blanes2009,Blanes2010,Mclachlan2002} for an exhaustive discussion of splitting methods for broad classes of ODEs and to \cite{Ableidinger2016,Ableidinger2017,Brehier2018,Leimkuhler2015,Milstein2004,Misawa2001,Moro2007,Petersen1998,Shardlow2003} for extensions to SDEs. Here, we derive two Lie-Trotter and two Strang splitting schemes for the IGBM. While the Lie-Trotter schemes coincide with the methods discussed in \cite{Moro2007}, the Strang schemes have not been considered before. 

The ODE approach \cite{Wong1965_2,Wong1965} is based on the idea of linking Stratonovich calculus with ODE tools. To construct higher-order schemes, this approach has been extended by defining the underlying ODE via a truncated exponential Lie series expansion, where iterated integrals of Brownian motion and time are approximated by their means, conditioned on the given increments of the Wiener process \cite{Castell1996,Malham2008}. Here, we consider the classical method \cite{Wong1965_2}, sometimes called piecewise linear method, and the scheme recently introduced by Foster et al. \cite{Foster2020}. They proposed a pathwise polynomial approximation method of the Brownian motion, which was used to estimate third order iterated integrals of Brownian motion and time. Incorporating these results into the ODE approach yielded a new numerical method for the IGBM, extending the classical ODE~method.

Among the properties of the IGBM, we are interested in both its conditional and asymptotic features (mean, variance and stationary density) and its boundary behaviour. The conditional and asymptotic mean and variance of this process are explicitly known. Hence, our first goal is to analyse whether the numerical methods accurately reproduce them. In particular, we derive closed-form expressions for the conditional and asymptotic means and variances of the  considered numerical methods. These quantities differ from the true ones. For this reason, we compare the resulting explicit biases. Knowing them is particularly relevant because it allows for a direct control of the respective simulation accuracy through the time discretisation step. This may be particularly beneficial, for example, in different statistical inference tools.

Other features we are interested in are the boundary properties of the IGBM. Depending on the parameter $\mu$, the IGBM possesses different properties at the boundary zero, according to Feller's classification \cite{Karlin1981}. Our second goal is to analyse whether the numerical methods preserve them. This is particularly important, since the nature of a boundary may force the process to change its behaviour near or at the boundary. While frequently applied numerical methods, such as the Euler-Maruyama, Milstein or higher-order It\^{o}-Taylor approximation schemes, may fail in meeting such conditions \cite{Alfonsi2005,Malham2013,Moro2007}, we prove that the splitting and ODE schemes preserve them. While Feller's boundary classification is a standard concept in the field of stochastic analysis, it is not so often adopted as a qualitative feature in the analysis of numerical methods. An exception constitutes the topic of positivity preservation, often studied in terms of the square-root process \cite{Alfonsi2005,Kahl2008,Malham2013,Misawa2001} and the  domain-invariance \cite{Gobet2001,Mackevicius2007,Pierret2016,Stamatiou2018}. For an investigation of these issues related to splitting methods, we refer to \cite{Mackevicius2007,Misawa2001}. For a discussion of Feller's classification in the context of splitting schemes, we refer to \cite{Moro2007}, where the focus lies on proving convergence results and only Lie-Trotter compositions are considered.

If the parameter $\mu=0$, the IGBM coincides with the well-known GBM \cite{Arnold1974,Mao2011}, which has often been used as a test equation in the field of stochastic linear stability analysis in the mean-square or almost sure sense  \cite{Buckwar2011,Higham2000,Saito1996}. This theory has been introduced by Mitsui and Saito \cite{Saito1993,Saito1996}, based on stability theory in the sense of Lyapunov \cite{Khasminskii2012}, and has been extended to systems of SDEs in \cite{Buckwar2010,Buckwar2012,Saito2002,Tocino2012}. Since the standard setting of this approach requires a constant equilibrium solution for which both the drift and diffusion components become zero, it cannot be applied to the IGBM. Nevertheless, known results for the Euler-Maruyama and Milstein schemes applied to the GBM are covered by our study as a special case. Thus, the results presented in this article are also related to stochastic stability analysis for SDEs with inhomogeneous drift coefficients.

The paper is organised as follows.
In Section \ref{sec2}, we introduce the IGBM and recall its properties. In Section \ref{sec3}, we provide a brief account of the splitting and ODE approaches, and introduce the considered numerical schemes for the IGBM. In Section \ref{sec4}, we provide closed-form expressions for the conditional and asymptotic means and variances of the investigated schemes, analyse the resulting biases and discuss the boundary preservation. In Section \ref{sec5}, we illustrate the theoretical results of Section~\ref{sec4} through a series of simulations. Moreover, we illustrate the strong (mean-square) convergence rates of the different numerical methods and investigate their required computational efforts. In addition, we analyse their ability to approximate the underlying stationary density and study their behaviour at the lower boundary. Conclusions are reported in~Section~\ref{sec6}.

\vspace{-0.3cm}
\section{The IGBM and its properties}
\label{sec2}
\vspace{-0.1cm}

The IGBM is described by the It\^o SDE
\begin{equation} \label{IGBM1}
dY(t)=\underbrace{\left(-\frac{1}{\tau}Y(t)+\mu\right)}_{:=F(Y(t))} dt+\underbrace{\sigma Y(t)}_{:=G(Y(t))} dW(t), \quad t \geq 0, \quad Y(0)=Y_0,
\end{equation}
where $\tau,\sigma>0$, $\mu \in \mathbb{R}$ and $W=(W(t))_{t \geq 0}$ is a standard Wiener process defined on the probability space $(\Omega,\mathcal{F},\mathbb{P})$ with a filtration $\mathcal{F}=(\mathcal{F}(t))_{t \geq 0}$ generated by $W$. The initial value $Y_0$ is either a deterministic non-negative constant or an $\mathcal{F}(0)$-measurable non-negative random variable with finite second moment. Since~\eqref{IGBM1} is a linear and autonomous SDE, a unique strong solution process $Y=(Y(t))_{t \geq 0}$ exists \cite{Arnold1974,Mao2011}. The solution of the homogeneous SDE (if $\mu=0$) corresponds to the well-known GBM. The solution of the inhomogeneous equation can be expressed in terms of the embedded GBM. In particular, applying the variation of constants formula \cite{Mao2011} to \eqref{IGBM1} yields
\begin{equation}\label{Expl_Sol}
Y(t)=e^{-(\frac{1}{\tau}+\frac{\sigma^2}{2})t+\sigma W(t)} \left( Y_0+{\mu} \int_{0}^{t} e^{(\frac{1}{\tau}+\frac{\sigma^2}{2})s-\sigma W(s) } ds \right).
\end{equation}

\paragraph{Conditional and asymptotic mean and variance} 
Since $Y_0$ has finite second moment, the mean and variance of the strong solution process $Y$, conditioned on the initial value $Y_0$, exist. They are explicitly known \cite{Barone2005,Donofrio2018,Zhao2009} and given by 
\begin{eqnarray}
	\label{cond_mean}\mathbb{E}[Y(t)|Y_0]&=&Y_0 e^{-\frac{1}{\tau}t}+\mu\tau (1-e^{-\frac{1}{\tau}t}),\vspace{0.2cm} \\ 
	\hspace{0.9cm}\label{cond_var}\textrm{Var}(Y(t)|Y_0)&=&\begin{cases}
		e^{-\frac{1}{\tau}t}\left( 2\mu[t Y_0 -\tau Y_0 -t\mu\tau]+Y_0^2 \right)\\ \hspace{0.5cm}-e^{-\frac{2}{\tau}t} (Y_0-\mu\tau)^2 + (\mu\tau)^2, &  \text{if} \ \sigma^2\tau=1, \vspace{0.2cm} \\ 
		e^{-\frac{1}{\tau}t}\left( 4\mu\tau[\mu\tau - Y_0]  \right) - e^{-\frac{2}{\tau}t} \left( Y_0-\mu\tau \right)^2 \\ \hspace{0.5cm} + 2\mu^2\tau t-3(\mu\tau)^2+ 2\mu\tau Y_0 +Y_0^2, &  \text{if} \ \sigma^2\tau=2, \vspace{0.2cm} \\
		\frac{(\mu\tau)^2\sigma^2\tau}{2-\sigma^2\tau}+2\tau\sigma^2 \frac{(Y_0-\mu\tau)\mu\tau}{1-\sigma^2\tau}e^{-\frac{1}{\tau}t}-e^{-\frac{2}{\tau}t}(Y_0-\mu\tau)^2 \\
		\hspace{0.5cm}+e^{(\sigma^2-\frac{2}{\tau})t}\left[ Y_0^2 - \frac{2Y_0\mu\tau}{1-\sigma^2\tau}+\frac{2 (\mu\tau)^2}{(2-\sigma^2\tau)(1-\sigma^2\tau)} \right], &  \text{otherwise}.
	\end{cases}
\end{eqnarray}

Since $\tau>0$, from \eqref{cond_mean}, it follows that the asymptotic mean of $Y$ exists. It is given by \vspace{-0.1cm}
\begin{equation}\label{statE}
\mathbb{E}[Y_\infty]:=\lim\limits_{t \to \infty} \mathbb{E}[Y(t)|Y_0]=\mu\tau.
\end{equation}
From \eqref{cond_var}, it follows that, under the condition $\sigma^2\tau<2$, the asymptotic variance of $Y$ exists. It is given by
\begin{equation}\label{statV}
\textrm{Var}(Y_\infty):=\lim\limits_{t \to \infty} \textrm{Var}(Y(t)|Y_0)=\frac{(\mu\tau)^2}{\frac{2}{\sigma^2\tau}-1}.
\end{equation}\noindent

\vspace{-0.3cm}
\paragraph{Boundary properties}

Depending on the parameter $\mu$, the IGBM possesses different properties at the boundary~$0$ according to Feller's boundary classification \cite{Karlin1981}. In particular, if $\mu=0$ and $Y_0>0$, the boundary $0$ is unattainable and attracting, i.e., the process cannot reach $0$ in finite time, but is attracted to it as time tends to infinity. In terms of linear stochastic stability analysis, this means that the equilibrium solution $0$ is asymptotically almost sure stable, since $\mathbb{P}(\lim_{t \to \infty} Y(t)=0|Y_0>0)=1$. In the case that $\mu=Y_0=0$, the process is absorbed at the boundary immediately. If $\mu>0$, then $0$ is an entrance boundary, i.e., the process cannot reach the boundary in finite time if $Y_0>0$ or it immediately leaves $0$ and stays above it if $Y_0=0$. If $\mu<0$, the boundary is of exit type, i.e., the process can reach the boundary in finite time and, as soon as it attains the boundary, it leaves $[0,+\infty)$ and cannot return into it. In many applications the process is stopped when it reaches an exit boundary, such that its state space is $[0,+\infty)$. 

Feller's boundary classification is based on the idea of transforming the one-dimensional diffusion into a Wiener process, first by a change of space (through the scale density) and second by a change of time (through the speed density). The scale and speed densities are given by
\begin{eqnarray*}
	s(y)&:=&e^{-\int\limits_{y_0}^{y} \frac{2F(z)}{G^2(z) } \ dz }=s_0 e^{{2\mu}/{\sigma^2 y}}y^{{2}/{\sigma^2\tau}}, \quad s_0=e^{{-2\mu}/{\sigma^2y_0}}y_0^{{-2}/{\sigma^2\tau}}, \\ m(y)&:=&\frac{1}{G^2(y)s(y)}=m_0y^{-(2+{2}/{\sigma^2\tau})}e^{-{2\mu}/{\sigma^2 y}}, \quad m_0=s_0^{-1}\sigma^{-2},
\end{eqnarray*}
respectively, where $y_0>0$ and $F$ and $G$ denote the drift and diffusion coefficients defined in \eqref{IGBM1}. Further, the scale function is defined by
\vspace{-0.1cm}\begin{equation*}
S[x_0,x]:=\int_{x_0}^{x} s(y) \ dy, \vspace{-0.1cm}
\end{equation*}
where $x_0>0$. For the IGBM, the nature of the boundary $0$ is uniquely determined by the three quantities
\vspace{-0.2cm}\begin{equation*}
S(0,x]:=\lim\limits_{x_0 \to 0} S[x_0,x], \quad	\Sigma_0:=\int_{0}^{\epsilon} S(0,x]m(x) \ dx, \quad N_0:=\int_{0}^{\epsilon} S[x,\epsilon]m(x) \ dx,
\end{equation*}
for an arbitrary $\epsilon>0$. If $\mu=0$, then $S(0,x]<\infty$, $\Sigma_0=\infty$ and $N_0=\infty$.
If $\mu>0$, then $S(0,x]=\infty$, $\Sigma_0=\infty$ and $N_0<\infty$. If $\mu<0$, then $S(0,x]<\infty$, $\Sigma_0<\infty$ and $N_0=\infty$. This implies the different types of boundary behaviour explained above, see Table 6.2 in \cite{Karlin1981}. According to this classification, we define the following properties, which are satisfied by the IGBM:
\begin{itemize}
	\item[$\bullet$] \textit{Unattainable property}: If $\mu\geq0$, then $\mathbb{P}(Y(t)>0 \ \forall \ t \geq 0 | Y_0>0 )=1$.
	\item[$\bullet$] \textit{Absorbing property}: If $\mu=0$, then $\mathbb{P}(Y(t)=0 \ \forall \ t \geq 0 | Y_0=0 )=1$.
	\item[$\bullet$] \textit{Entrance property}: If $\mu>0$, then $\mathbb{P}(Y(t)>0 \ \forall \ t > 0 | Y_0=0 )=1$.
	\item[$\bullet$] \textit{Exit property}: If $\mu < 0$, then $\mathbb{P}(Y(t)<0 \ \forall \ t > s | Y(s)\leq 0 )=1$.
\end{itemize}

\vspace{-0.5cm}
\section{Numerical methods for the IGBM}
\label{sec3}
\vspace{-0.2cm}

Consider a discretised time interval $[0,t_{\text{max}}]$, $t_{\text{max}}>0$, with equidistant time steps $\Delta=t_i-t_{i-1}$, $i=1,...,N$, $N \in \mathbb{N}$, $t_0=0$ and $t_N=t_{\text{max}}$. We denote by $\widetilde{Y}(t_i)$ a numerical realisation of the process $Y$ at the discrete time points $t_i=i\Delta$, where $\widetilde{Y}(t_0):=Y_0$. Moreover, we denote by $\xi_{i-1}:=W(t_i)-W(t_{i-1})\sim N(0,\Delta)$, $i=1,...,N$, the Wiener increments which are independent and identically distributed (iid) normal random variables with null mean and variance $\Delta$. In the following, we recall different numerical methods used to generate values $\widetilde{Y}(t_i)$ of the IGBM.

\subsection{It\^{o}-Taylor expansion approach}
\label{subsec:3:0:new}

The most popular approach to derive numerical methods for SDEs is to use appropriate truncations of the It\^o-Taylor series expansion \cite{Kloeden1992,Milstein2004}.

\subsubsection{Euler-Maruyama and Milstein schemes}
\label{subsec:3:0}
\vspace{-0.1cm}

Two of the most well-known methods in this class are the Euler-Maruyama and the Milstein schemes. The Euler-Maruyama method yields trajectories of the IGBM through the iteration
\begin{equation}\label{EM}
\widetilde{Y}^{\textrm{E}}(t_i)=\widetilde{Y}^{\textrm{E}}(t_{i-1}) + \Delta  \left( -\frac{1}{\tau}\widetilde{Y}^{\textrm{E}}(t_{i-1}) +\mu \right) + \sigma \widetilde{Y}^{\textrm{E}}(t_{i-1})\xi_{i-1}.
\end{equation}
This method is mean-square convergent of order $1/2$. This rate can be increased by taking into account additional terms of the It\^{o}-Taylor expansion. In particular, the Milstein method yields trajectories of the IGBM via
\begin{equation}\label{M}
\widetilde{Y}^{\textrm{M}}(t_i)=\widetilde{Y}^{\textrm{M}}(t_{i-1}) + \Delta  \left( -\frac{1}{\tau}\widetilde{Y}^{\textrm{M}}(t_{i-1}) +\mu \right) + \sigma \widetilde{Y}^{\textrm{M}}(t_{i-1}) \left(\xi_{i-1}+\frac{\sigma}{2}(\xi_{i-1}^2-\Delta) \right),
\end{equation}
and has a mean-square convergence rate of order $1$.

\subsection{Splitting approach}
\label{subsec:3:1}

The second approach we focus on is based on splitting methods \cite{Blanes2009,Hairer2006,Mclachlan2002,Misawa2001}. A brief account of their key ideas is provided in the following. Consider an It\^o SDE of the form
\begin{equation}\label{dot_f}
dY(t)=F(Y(t)) dt+G(Y(t)) dW(t), \quad t \geq 0, \quad Y(0)=Y_0,
\end{equation}
where the drift coefficient and the diffusion component can be expressed as
\begin{equation*}
F(Y(t))=\sum_{l=1}^{d}F^{[l]}(Y(t)), \quad G(Y(t))=\sum_{l=1}^{d}G^{[l]}(Y(t)), \quad d \in \mathbb{N}.
\end{equation*}
Usually, there are several ways how to decompose the components $F$ and $G$. The goal is to obtain subequations
\begin{equation}\label{Sub}
dY^{[l]}(t)=F^{[l]}(Y^{[l]}(t)) dt+G^{[l]}(Y^{[l]}(t)) dW(t), \quad l \in  \{ 1,...,d \},
\end{equation}
which can be solved explicitly. Once the explicit solutions are derived, they need to be composed. Two common procedures for doing this are the Lie-Trotter \cite{Trotter1959} and the Strang \cite{Strang1968} approach. Let~$\varphi_t^{[l]}(Y_0)$ denote the exact flows (solutions) of the subequations in \eqref{Sub} at time $t$ and starting from $Y_0$. Then, the Lie-Trotter composition of flows
\begin{equation*}
\widetilde{Y}(t_i)=\left( \varphi_\Delta^{[1]} \circ ... \circ \varphi_\Delta^{[d]} \right)(\widetilde{Y}(t_{i-1}))
\end{equation*}
and the Strang approach
\begin{equation*}
\widetilde{Y}(t_i)=\left( \varphi_{\Delta/2}^{[1]} \circ ... \circ \varphi_{\Delta/2}^{[d-1]} \circ \varphi_{\Delta}^{[d]} \circ \varphi_{\Delta/2}^{[d-1]} \circ ... \circ \varphi_{\Delta/2}^{[1]} \right)(\widetilde{Y}(t_{i-1}))
\end{equation*}
yield numerical methods for \eqref{dot_f}. The order of the evaluations of the exact flows can be changed, yielding different schemes within each approach.

\subsubsection{Lie-Trotter and Strang schemes for the IGBM}

With the purpose of excluding the inhomogeneous part, relying thus on the underlying GBM, we split \eqref{IGBM1} into two simple subequations, namely
\begin{equation}\label{SDE}
dY^{[1]}(t)=\underbrace{-\frac{1}{\tau}Y^{[1]}(t)}_{F^{[1]}(Y^{[1]}(t))}dt+\underbrace{\sigma Y^{[1]}(t)}_{G^{[1]}(Y^{[1]}(t))}dW(t),
\end{equation}
\begin{equation}\label{ODE}
dY^{[2]}(t)=\underbrace{{\mu}}_{F^{[2]}}dt, \quad G^{[2]}\equiv 0.
\end{equation}
The first equation, corresponding to the GBM, allows for an exact simulation of sample paths through
\begin{equation}\label{eq1}
Y^{[1]}(t_i)=\varphi_{\Delta}^{[1]}(Y^{[1]}(t_{i-1}))=Y^{[1]}(t_{i-1}) e^{-(\frac{1}{\tau}+\frac{\sigma^2}{2})\Delta + \sigma \xi_{i-1}}, \quad i=1,\ldots, N.
\end{equation}
The second equation is a simple ODE with its explicit solution given by
\begin{equation}\label{eq2}
Y^{[2]}(t_{i})=\varphi_{\Delta}^{[2]}(Y^{[2]}(t_{i-1}))=Y^{[2]}(t_{i-1})+{\mu} \Delta, \quad i=1,\ldots, N.
\end{equation}
The Lie-Trotter composition yields 
\begin{eqnarray}
\hspace{1.0cm}\label{SP1}\widetilde{Y}^{\textrm{L1}}(t_i)&:=&\left( \varphi_\Delta^{[1]} \circ \varphi_\Delta^{[2]} \right)(\widetilde{Y}^{\textrm{L1}}(t_{i-1}))=e^{-(\frac{1}{\tau}+\frac{\sigma^2}{2})\Delta+\sigma \xi_{i-1}}  \left( \widetilde{Y}^{\textrm{L1}}(t_{i-1})+{\mu}  \Delta \right), \\ 
\label{SP2}\widetilde{Y}^{\textrm{L2}}(t_i)&:=&\left( \varphi_\Delta^{[2]} \circ \varphi_\Delta^{[1]} \right)(\widetilde{Y}^{\textrm{L2}}(t_{i-1}))=\widetilde{Y}^{\textrm{L2}}(t_{i-1})e^{-(\frac{1}{\tau}+\frac{\sigma^2}{2})\Delta+\sigma \xi_{i-1}}   +{\mu} \Delta,
\end{eqnarray}
and the Strang approach results in 
\begin{eqnarray}
\hspace{-0.8cm}\label{SP3}\widetilde{Y}^{\textrm{S1}}(t_i)&:=&\left( \varphi_{\Delta/2}^{[2]} \circ \varphi_{\Delta}^{[1]} \circ \varphi_{\Delta/2}^{[2]} \right)(\widetilde{Y}^{\textrm{S1}}(t_{i-1})) =\left(\widetilde{Y}^{\textrm{S1}}(t_{i-1}) +{\mu}  \frac{\Delta}{2} \right)e^{-(\frac{1}{\tau}+\frac{\sigma^2}{2})\Delta+\sigma \xi_{i-1}}+{\mu}  \frac{\Delta}{2}, \\ 
\hspace{-0.8cm}\label{SP4}\widetilde{Y}^{\textrm{S2}}(t_i)&:=&\left( \varphi_{\Delta/2}^{[1]} \circ \varphi_{\Delta}^{[2]} \circ \varphi_{\Delta/2}^{[1]} \right)(\widetilde{Y}^{\textrm{S2}}(t_{i-1})) \\ \nonumber  &=& \widetilde{Y}^{\textrm{S2}}(t_{i-1})e^{-(\frac{1}{\tau}+\frac{\sigma^2}{2})\Delta+\sigma(\varphi_{i-1}+\psi_{i-1})} + {\mu}\Delta e^{-(\frac{1}{\tau}+\frac{\sigma^2}{2})\frac{\Delta}{2}+\sigma \psi_{i-1}},
\end{eqnarray}
with iid random variables $\varphi_{i-1}$, $\psi_{i-1} \sim\mathcal{N}(0,{\Delta}/{2})$. The equations \eqref{SP1}-\eqref{SP4} define four different numerical solutions of \eqref{IGBM1}. For a discussion of the mean-square convergence of the second Lie-Trotter method \eqref{SP2} we refer to \cite{Moro2007}, where a rate of order $1$ has been proved. It is expected that this result extends to the other three splitting schemes, a conjecture that we confirm experimentally in Subsection \ref{sec5:1}. In particular, it has been observed that, in contrast to the deterministic case \cite{Hairer2006}, the convergence rate of splitting schemes for SDEs cannot be increased by using Strang compositions, i.e., compositions based on fractional $\Delta/2$ steps~\cite{Milstein2003}.

\newpage

\subsection{ODE approach}
\label{subsec:3:new}

An alternative approach to derive numerical solutions of SDEs is to solve properly derived ODEs, a methodology that we briefly recall in the following. Consider the Stratonovich version of \eqref{dot_f} given by
\begin{equation}\label{SDE_Stratonovich}
dY(t)=\bar{F}(Y(t))dt+G(Y(t)) \circ dW(t), \quad t\geq 0, \quad Y(0)=Y_0,
\end{equation}
where 
\begin{equation*}
\bar{F}(y)=F(y)-\frac{1}{2}G(y)G'(y),
\end{equation*}
with $G'(y)$ denoting the derivative of $G$ with respect to $y$.
Then, given a fixed time step $\Delta>0$ and a Wiener increment $\xi_{i-1}$, a numerical solution $\widetilde{Y}(t_i)$ of SDE \eqref{SDE_Stratonovich} can be obtained by defining it as the solution at $u=1$ of the ODE
\begin{equation}\label{ODE_lin}
\frac{dz}{du}=\bar{F}(z)\Delta+G(z)\xi_{i-1}, \quad z_0=\widetilde{Y}(t_{i-1}).
\end{equation}
This method has been observed to have a mean-square convergence rate of order $1$, see, e.g., \cite{Castell1996,Foster2020}, and is called piecewise linear method, since it uses piecewise linear approximations of Brownian paths.

Recently, Foster et al. \cite{Foster2020} proposed an extended variant of this approach, using polynomial approximations of Brownian motion. This yielded numerical schemes for SDEs with mean-square order~$1.5$. In particular, a numerical solution $\widetilde{Y}(t_i)$ of SDE \eqref{SDE_Stratonovich} can be obtained by defining it as the solution at $u=1$ of the ODE
\begin{equation}\label{ODE_log}
\frac{dz}{du}=\bar{F}(z)\Delta+G(z)\xi_{i-1} + [G,\bar{F}](z)\Delta \rho_{i-1} + \bigl[G, [G,\bar{F}] \bigr](z) \left( \frac{3}{5}\Delta \rho_{i-1}^2 + \frac{\Delta^2}{30} \right), \quad z_0=\widetilde{Y}(t_{i-1}), 
\end{equation} 
where $[\cdot,\cdot]$ denotes the standard Lie bracket of vector fields, and the
\begin{equation*}
\rho_{i-1}:=\frac{1}{\Delta}\int\limits_{t_{i-1}}^{t_i} \left[ W(u)-W(t_{i-1})-\frac{u-t_{i-1}}{\Delta} \Bigl( W(t_i)-W(t_{i-1}) \Bigr) \right] \ du \vspace{-0.1cm}
\end{equation*}
are rescaled space-time L\'{e}vy areas of the Wiener process over  $[t_{i-1},t_i]$. They are shown to have distribution $\rho_{i-1}\sim N(0,\Delta/12)$ and to be independent of the Wiener increments $\xi_{i-1}$. Following the notion in \cite{Foster2020}, we call this method log-ODE scheme, and we refer to \cite{Foster2020} for further details.

\vspace{-0.1cm}
\subsubsection{Piecewise linear and log-ODE schemes for the IGBM}
\vspace{-0.1cm}

To derive numerical schemes for the IGBM based on the ODE approach, consider the Stratonovich version of SDE \eqref{IGBM1} given by
\begin{equation*}
dY(t)=\left( -\Bigl( \frac{1}{\tau}+\frac{\sigma^2}{2} \Bigr)Y(t) + \mu \right)dt + \sigma Y(t) \circ dW(t), \quad t\geq 0, \quad Y(0)=Y_0.
\end{equation*}
Solving the corresponding ODE \eqref{ODE_lin} yields the following piecewise linear scheme
\begin{equation}\label{Lin}
\widetilde{Y}^{\textrm{Lin}}(t_i)=\widetilde{Y}^{\textrm{Lin}}(t_{i-1})e^{-(\frac{1}{\tau}+\frac{\sigma^2}{2})\Delta+\sigma \xi_{i-1}}+\mu\Delta \left( \frac{e^{-(\frac{1}{\tau}+\frac{\sigma^2}{2})\Delta+\sigma \xi_{i-1}}-1}{-(\frac{1}{\tau}+\frac{\sigma^2}{2})\Delta+\sigma \xi_{i-1}} \right).
\end{equation}
Noting that 
\hspace{-0.5cm}\begin{eqnarray*}
	[G,\bar{F}](y)&=&\bar{F}'(y)G(y)-G'(y)\bar{F}(y)=-\mu\sigma,\\
	\bigl[G,[G,\bar{F}]\bigr](y)&=&\mu\sigma^2,
\end{eqnarray*}
and solving the respective ODE \eqref{ODE_log} yields the following log-ODE scheme for the IGBM \cite{Foster2020}
\begin{eqnarray}\label{Log}
\nonumber	\widetilde{Y}^{\textrm{Log}}(t_i)&=&\widetilde{Y}^{\textrm{Log}}(t_{i-1})e^{-(\frac{1}{\tau }+\frac{\sigma^2}{2})\Delta+\sigma \xi_{i-1}} \\ &&+\mu\Delta \left( \frac{e^{-(\frac{1}{\tau}+\frac{\sigma^2}{2})\Delta+\sigma \xi_{i-1}}-1}{-(\frac{1}{\tau}+\frac{\sigma^2}{2})\Delta+\sigma \xi_{i-1}} \right)\left( 1-\sigma \rho_{i-1} +\sigma^2 \Bigl( \frac{3}{5} \rho_{i-1}^2 + \frac{\Delta}{30} \Bigr) \right). 
\end{eqnarray}

\begin{remark}
	The numerical solutions \eqref{SP1}-\eqref{SP4} coincide with 
	the discretised version of \eqref{Expl_Sol}, where the integral is approximated using the left point rectangle rule, the right point rectangle rule, the trapezoidal rule and the midpoint rule, respectively. If $\mu=0$, the numerical solutions \eqref{SP1}-\eqref{SP4} and \eqref{Lin}, \eqref{Log} coincide with the exact simulation scheme \eqref{eq1} for the GBM. 
\end{remark}
\begin{notation}
	In the following, we use the abbreviations E, M, L1, L2, S1, S2, Lin and Log for the Euler-Maruyama~\eqref{EM}, Milstein \eqref{M},  first Lie-Trotter \eqref{SP1}, second Lie-Trotter \eqref{SP2}, first Strang \eqref{SP3}, second Strang \eqref{SP4}, piecewise linear \eqref{Lin} and log-ODE \eqref{Log} methods, respectively.
\end{notation}

\section{Properties of the numerical methods for the IGBM}
\label{sec4}

We now examine the ability of the derived numerical methods to accurately preserve the properties of the process. In particular, we first provide closed-form expressions for their conditional and asymptotic means and variances and analyse the resulting biases. Then, we show that the four splitting and the two ODE schemes preserve the boundary properties of the IGBM, while the Euler-Maruyama and Milstein schemes do not.

\subsection{Investigation of the conditional moments}
\label{subsec:4:0}

The numerical solutions defined by \eqref{EM}, \eqref{M},  \eqref{SP1}-\eqref{SP4}, \eqref{Lin} and \eqref{Log} enable to express $\widetilde{Y}(t_i)$ in terms of the initial value $Y_0$. Indeed, by performing back iteration, we obtain
\begin{eqnarray}
\label{Ye} \hspace{-0.9cm} \widetilde{Y}^{\textrm{E}}(t_i)&=&Y_0 \prod\limits_{j=1}^{i}\left(1-\frac{\Delta}{\tau}+\sigma\xi_{i-j}\right)+\mu\Delta\sum\limits_{k=1}^{i-1}\prod\limits_{j=1}^{k}\left(1-\frac{\Delta}{\tau}+\sigma\xi_{i-j}\right) + \mu\Delta,\\ 
\label{Ym} \hspace{-0.9cm} \nonumber \widetilde{Y}^{\textrm{M}}(t_i)&=&Y_0 \prod\limits_{j=1}^{i}\left(1-\frac{\Delta}{\tau}+\sigma\xi_{i-j}+(\xi_{i-j}^2-\Delta)\frac{\sigma^2}{2}\right) \\   &&+\mu\Delta\sum\limits_{k=1}^{i-1}\prod\limits_{j=1}^{k}\left(1-\frac{\Delta}{\tau}+\sigma\xi_{i-j} +(\xi_{i-j}^2-\Delta)\frac{\sigma^2}{2} \right) + \mu\Delta,\\
\label{Y1} \hspace{-0.9cm} \widetilde{Y}^{\textrm{L1}}(t_i)&=&Y_0 e^{-(\frac{1}{\tau}+\frac{\sigma^2}{2})t_i+\sigma \sum\limits_{k=0}^{i-1}\xi_k}+{\mu}\Delta \sum\limits_{k=1}^{i} e^{-(\frac{1}{\tau}+\frac{\sigma^2}{2})t_k+\sigma \sum\limits_{j=1}^{k}\xi_{i-j}},\\
\label{Y2} \hspace{-0.9cm} \widetilde{Y}^{\textrm{L2}}(t_i)&=&Y_0 e^{-(\frac{1}{\tau}+\frac{\sigma^2}{2})t_i+\sigma \sum\limits_{k=0}^{i-1}\xi_k}+{\mu}\Delta \sum\limits_{k=0}^{i-1} e^{-(\frac{1}{\tau}+\frac{\sigma^2}{2})t_k+\sigma \sum\limits_{j=1}^{k}\xi_{i-j}}, \\ 
\label{Y3} \hspace{-0.9cm} \widetilde{Y}^{\textrm{S1}}(t_i)&=&\left(Y_0+\frac{\mu\Delta}{2}\right)  e^{-(\frac{1}{\tau}+\frac{\sigma^2}{2})t_i+\sigma \sum\limits_{k=0}^{i-1}\xi_k}+{\mu}\Delta  \sum\limits_{k=1}^{i-1} e^{-(\frac{1}{\tau}+\frac{\sigma^2}{2})t_k+\sigma \sum\limits_{j=1}^{k}\xi_{i-j}}
+\frac{{\mu}\Delta }{2},\\
\label{Y4} \hspace{-0.9cm} \widetilde{Y}^{\textrm{S2}}(t_i)&=&Y_0 e^{-(\frac{1}{\tau}+\frac{\sigma^2}{2})t_i+\sigma \sum\limits_{k=0}^{i-1}\xi_k}+{\mu}\Delta \sum\limits_{k=1}^{i} e^{-(\frac{1}{\tau}+\frac{\sigma^2}{2})(k-\frac{1}{2})\Delta+\sigma \psi_{i-k}+\sigma \sum\limits_{j=1}^{k-1}\xi_{i-j}}, \\
\hspace{-0.9cm}\label{YLin} \widetilde{Y}^{\textrm{Lin}}(t_i)&=&Y_0 e^{-(\frac{1}{\tau}+\frac{\sigma^2}{2})t_i+\sigma \sum\limits_{k=0}^{i-1}\xi_k}+{\mu}\Delta \sum\limits_{k=0}^{i-1} e^{-(\frac{1}{\tau}+\frac{\sigma^2}{2})t_k+\sigma \sum\limits_{j=1}^{k}\xi_{i-j}}  \left( \frac{e^{-(\frac{1}{\tau}+\frac{\sigma^2}{2})\Delta+\sigma \xi_{i-1-k}}-1}{-(\frac{1}{\tau}+\frac{\sigma^2}{2})\Delta+\sigma \xi_{i-1-k}} \right), \\ 
\label{YLog} \hspace{-0.9cm} \nonumber \widetilde{Y}^{\textrm{Log}}(t_i)&=&Y_0 e^{-(\frac{1}{\tau}+\frac{\sigma^2}{2})t_i+\sigma \sum\limits_{k=0}^{i-1}\xi_k} \\ && \hspace{-2.45cm} + {\mu}\Delta \sum\limits_{k=0}^{i-1} e^{-(\frac{1}{\tau}+\frac{\sigma^2}{2})t_k+\sigma \sum\limits_{j=1}^{k}\xi_{i-j}}  \left( \frac{e^{-(\frac{1}{\tau}+\frac{\sigma^2}{2})\Delta+\sigma \xi_{i-1-k}}-1}{-(\frac{1}{\tau}+\frac{\sigma^2}{2})\Delta+\sigma \xi_{i-1-k}} \right) \left(1-\sigma \rho_{i-1-k} +\sigma^2\left[ \frac{3}{5} \rho_{i-1-k}^2 + \frac{\Delta}{30} \right] \right),
\end{eqnarray}
where $\xi_i:=\varphi_i+\psi_i$ in \eqref{Y4}.
These relations allow for an investigation of the conditional means $\mathbb{E}[\widetilde{Y}(t_i)|Y_0]$ and variances $\textrm{Var}(\widetilde{Y}(t_i)|Y_0)$ of the numerical solutions.

\subsubsection{Closed-form expressions for the conditional means and variances}
\label{subsec:4:0:1}

In Proposition \ref{Prop_cond_1}, we provide closed-form expressions of the conditional mean and variance of a general random variable $Z_i$ that plays the role of a numerical solution $\widetilde{Y}(t_i)$ as in \eqref{Ye}-\eqref{YLog} for a fixed time $t_i$. These expressions will allow for a straightforward derivation of the corresponding results for the numerical solutions of interest.

\begin{proposition}\label{Prop_cond_1}
	Consider the real-valued random variable $Z_i$ defined by
	\begin{equation}
	\label{Zi}Z_i:= Z_0 W_i+c_1\sum_{k=0}^{I}W_k H_{k+1}+c_2,
	\end{equation}
	where $i \in \mathbb{N}$, $I \in \{ i-1,i \}$, $c_1,c_2>0$, $Z_0 \in \mathbb{R}$, $W_0 \in \{ 0,1 \}$, $W_k:=\prod\limits_{j=1}^k X_j$ with $X_j$, \text{$j=1,\ldots,k$}, being iid with mean $\mu_x \in \mathbb{R}$ and second moment $r>0$. The $H_{k+1}$, $k=0,\ldots,I$, are iid with mean $\mu_h \in \mathbb{R}$ and second moment $r_h>0$. Moreover, $W_k$ and $H_{k+1}$ are independent and \text{$\mathbb{E}[W_lW_kH_{k+1}]=r^k\mu_x^{l-k}p$}, for $k<l$ and $p\in \mathbb{R}$. The mean of $Z_i$ conditioned on $Z_0$ is given by 
	\begin{eqnarray}\label{EZi}
	\mathbb{E}[Z_i|Z_0]=Z_0\mu_x^i+c_1 \mu_h \sum_{k=1}^{I} \mu_x^k+c_1  W_0 \mu_h+c_2
	\end{eqnarray}
	and the variance of $Z_i$ conditioned on $Z_0$ is given by 
	\begin{eqnarray}\label{VZi}
	\textrm{Var}(Z_i|Z_0)
	\nonumber&=& Z_0^2(r^i-\mu_x^{2i})+2c_1 Z_0\sum_{k=0}^I  r^k\mu_x^{i-k}p-\mu_x^{i+k}\mu_h  \\ &&\hspace{0.5cm}+c_1^2\left[\sum_{k=0}^I r^{k}r_h-\mu_x^{2k}\mu_h^2+2\sum_{l=1}^I\sum_{k=0}^{l-1} \mu_h r^k\mu_x^{l-k}p-\mu_x^{l+k}\mu_h^2 \right].
	\end{eqnarray}
\end{proposition}\noindent
The proof of Proposition \ref{Prop_cond_1} is given in Appendix \ref{appA}.

Based on Proposition \ref{Prop_cond_1}, we derive the conditional moments of the Euler-Maruyama, Milstein, splitting and ODE schemes.
\begin{corollary}\label{Lemma_cond_1}
	Let $\widetilde{Y}(t_i)$ be the numerical solutions defined through \eqref{EM}, \eqref{M}, \eqref{SP1}-\eqref{SP4}, \eqref{Lin} and \eqref{Log}, respectively, at time $t_i=i\Delta$. Their means and variances conditioned on the initial value $Y_0$ are given by \eqref{EZi} and \eqref{VZi}, respectively, with quantities $\mu_x$, $\mu_h$, $r$, $r_h$, $p$, $c_1$, $c_2$, $I$, $Z_0$ and $W_0$ defined as reported in Table \ref{table1}.
\end{corollary}\noindent
The proof of Corollary \ref{Lemma_cond_1} is given in Appendix \ref{appB}.

\begin{remark}\label{rem_2}
	To make the results of Proposition \ref{Prop_cond_1} and Corollary \ref{Lemma_cond_1} more approachable, the conditional means of the considered numerical methods are listed in closed-form as follows
	\begin{eqnarray*}
		\mathbb{E}\left[ \widetilde{Y}^{\textrm{E}}(t_i)|Y_0 \right]&=& \mathbb{E}\left[ \widetilde{Y}^{\textrm{M}}(t_i)|Y_0 \right]= Y_0 \left( 1-\frac{\Delta}{\tau} \right)^i + \mu\Delta \left( \frac{1- \left( 1-\frac{\Delta}{\tau} \right)^i }{\Delta/\tau} \right),\\
		\mathbb{E}\left[ \widetilde{Y}^{\textrm{L1}}(t_i)|Y_0 \right]&=&  Y_0 e^{-\frac{1}{\tau}t_i} + \mu \tau \left( 1-e^{-\frac{1}{\tau}t_i} \right)  \left( \frac{\Delta/\tau}{e^{\Delta/\tau}-1} \right), \\
		\mathbb{E}\left[ \widetilde{Y}^{\textrm{L2}}(t_i)|Y_0 \right]&=&  Y_0 e^{-\frac{1}{\tau}t_i} + \mu \tau \left( 1-e^{-\frac{1}{\tau}t_i} \right)  \left( \frac{\Delta/\tau}{e^{\Delta/\tau}-1} \right)e^{\Delta/\tau}, \\
		\mathbb{E}\left[ \widetilde{Y}^{\textrm{S1}}(t_i)|Y_0 \right]&=&  Y_0 e^{-\frac{1}{\tau}t_i} + \mu \tau \left( 1-e^{-\frac{1}{\tau}t_i} \right)  \left[  \left(\frac{\Delta/\tau}{e^{\Delta/\tau}-1} \right) e^{\Delta/\tau} - \frac{\Delta}{2\tau} \right], \\
		\mathbb{E}\left[ \widetilde{Y}^{\textrm{S2}}(t_i)|Y_0 \right]&=&  Y_0 e^{-\frac{1}{\tau}t_i} + \mu \tau \left( 1-e^{-\frac{1}{\tau}t_i} \right)  \left( \frac{\Delta/\tau}{e^{\Delta/\tau}-1} \right)e^{\Delta/2\tau}, 
	\end{eqnarray*}
	\begin{eqnarray*}
		\mathbb{E}\left[ \widetilde{Y}^{\textrm{Lin}}(t_i)|Y_0 \right]&=&  Y_0 e^{-\frac{1}{\tau}t_i} + \mu \tau \left( 1-e^{-\frac{1}{\tau}t_i} \right)  \left( \frac{\Delta/\tau}{e^{\Delta/\tau}-1} \right)e^{\Delta/\tau}L_{\Delta,\tau,\sigma}, \\
		\mathbb{E}\left[ \widetilde{Y}^{\textrm{Log}}(t_i)|Y_0 \right]&=&  Y_0 e^{-\frac{1}{\tau}t_i} + \mu \tau \left( 1-e^{-\frac{1}{\tau}t_i} \right)  \left( \frac{\Delta/\tau}{e^{\Delta/\tau}-1} \right)e^{\Delta/\tau}L_{\Delta,\tau,\sigma}\left( 1+\sigma^2 \frac{\Delta}{12} \right), 
	\end{eqnarray*}
	where $L_{\Delta,\tau,\sigma}$ is defined as
	\begin{equation}\label{L_delta_tau_sigma}
	\hspace{-0.4cm}L_{\Delta,\tau,\sigma}:=
	\frac{\sqrt{\pi}}{\sigma \sqrt{2\Delta} } \exp \left( \frac{-\left( \frac{1}{\tau} +\frac{\sigma^2}{2} \right)^2\Delta}{2\sigma^2}  \right) \left( \text{erfi} \left[ \frac{\left( \frac{1}{\tau} + \frac{\sigma^2}{2} \right) \sqrt{\Delta} }{\sigma \sqrt{2}} \right] + \text{erfi} \left[ \frac{\left(- \frac{1}{\tau} + \frac{\sigma^2}{2} \right) \sqrt{\Delta} }{\sigma \sqrt{2}} \right]  \right),
	\end{equation} 
	with $\text{erfi}$ denoting the imaginary error function.
	The above expressions are obtained from \eqref{EZi} after calculating the geometric sums. Closed-form expression of the conditional variances can be obtained analogously. \noindent
\end{remark}

While the conditional means of the Euler-Maruyama and Milstein schemes are equal, their conditional variances are different. This results from the fact that the Milstein scheme takes into account an additional term that is related only to the diffusion coefficient of the SDE. Noting that $\mu\Delta=\mu\tau\Delta/\tau$, it can be observed that the conditional means of the Euler-Maruyama, Milstein and splitting methods depend on $\Delta/\tau$ and their conditional variances depend on $\Delta/\tau$ and $\Delta\sigma^2$. Remarkably, only the conditional means of the ODE methods depend on $\sigma$, while this is not the case for the true conditional mean \eqref{cond_mean}. If $\mu=0$, the conditional means and variances of the splitting schemes \eqref{SP1}-\eqref{SP4} and ODE schemes \eqref{Lin}, \eqref{Log} coincide with the true quantities \eqref{cond_mean} and \eqref{cond_var}, respectively, at time $t_i$.
\begin{remark}
	Having closed-form expressions for the conditional moments of the numerical solutions allows for a direct control of the simulation accuracy through the choice of the time step $\Delta$.
\end{remark}

\begin{table}[H] 
	{\footnotesize  
		\caption{Quantities of interest for the numerical schemes entering in \eqref{EZi}-\eqref{VZi}.}
		\label{table1}
		\begin{center}\vspace{-0.3cm}
			\scalebox{1.0}{
				\hskip-0.2cm\begin{tabular}{|c!{\vrule width 1pt}c|c|c|c|c|}
					\hline 
					$Z_i$ & $\mu_x$ & $\mu_h$ & $r$ & $r_h$ & $p$ \\ \noalign{\hrule height 1pt} 
					$\widetilde{Y}^{\textrm{E}}(t_i)$ & $1-\frac{\Delta}{\tau}$ & 1 & $\sigma^2\Delta+(1-\frac{\Delta}{\tau})^2$ & 1 &$1$  \\ 
					
					$\widetilde{Y}^{\textrm{M}}(t_i)$ & $1-\frac{\Delta}{\tau}$ & 1 &  $\sigma^2\Delta+(1-\frac{\Delta}{\tau})^2+\frac{(\sigma^2\Delta)^2}{2}$ & 1 & $1$  \\
					$\widetilde{Y}^{\textrm{L1}}(t_i)$  & $e^{-\Delta/\tau}$ & 1 & $e^{\sigma^2\Delta-2\Delta/\tau}$ & 1 & $1$ \\
					$\widetilde{Y}^{\textrm{L2}}(t_i)$ & $e^{-\Delta/\tau}$ & 1  & $e^{\sigma^2\Delta-2\Delta/\tau}$ & 1 & $1$  \\
					$\widetilde{Y}^{\textrm{S1}}(t_i)$  & $e^{-\Delta/\tau}$ & 1 & $e^{\sigma^2\Delta-2\Delta/\tau}$ & 1 & $1$  \\
					$\widetilde{Y}^{\textrm{S2}}(t_i)$ & $e^{-\Delta/\tau}$ & $\mu_x^{1/2}$ & $e^{\sigma^2\Delta-2\Delta/\tau}$ & $r^{1/2}$ & $r_h\mu_x^{-1/2}$  \\					
					$\widetilde{Y}^{\textrm{Lin}}(t_i)$ & $e^{-\Delta/\tau}$ & $L_{\Delta,\tau,\sigma}$ \eqref{L_delta_tau_sigma} & $e^{\sigma^2\Delta-2\Delta/\tau}$ & $\bar{L}_{\Delta,\tau,\sigma}$ \eqref{barL_delta_tau_sigma} & $\widetilde{L}_{\Delta,\tau,\sigma}$ \eqref{tildeL_delta_tau_sigma}  \\
					
					$\widetilde{Y}^{\textrm{Log}}(t_i)$ & $e^{-\Delta/\tau}$ & $K_{\Delta,\tau,\sigma}$ \eqref{K_delta_tau_sigma} & $e^{\sigma^2\Delta-2\Delta/\tau}$ & $\bar{K}_{\Delta,\tau,\sigma}$ \eqref{barK_delta_tau_sigma} & $\widetilde{K}_{\Delta,\tau,\sigma}$ \eqref{tildeK_delta_tau_sigma}  \\
					\hline
					\hline 
					$Z_i$ & $c_1$ & $c_2$ & $I$ & $Z_0$ & $W_0$ \\ \noalign{\hrule height 1pt}  
					$\widetilde{Y}^{\textrm{E}}(t_i)$ & $\mu\Delta$ & $\mu\Delta$ & $i-1$ & $Y_0$ & 0 \\
					$\widetilde{Y}^{\textrm{M}}(t_i)$ & $\mu\Delta$ & $\mu\Delta$ & $i-1$ & $Y_0$ & 0 \\
					$\widetilde{Y}^{\textrm{L1}}(t_i)$ & $\mu\Delta$ & 0 & $i$ & $Y_0$ & 0 \\
					$\widetilde{Y}^{\textrm{L2}}(t_i)$ & $\mu\Delta$ & $0$ & $i-1$ & $Y_0$ & $1$ \\
					$\widetilde{Y}^{\textrm{S1}}(t_i)$ & $\mu\Delta$ & $\frac{\mu\Delta}{2}$ & $i-1$ & $Y_0+\frac{\mu\Delta}{2}$ & 0 \\
					$\widetilde{Y}^{\textrm{S2}}(t_i)$ & $\mu\Delta$ & 0 & $i-1$ & $Y_0$ & 1 \\					
					$\widetilde{Y}^{\textrm{Lin}}(t_i)$ & $\mu\Delta$ & 0 & $i-1$ & $Y_0$ & 1 \\
					
					$\widetilde{Y}^{\textrm{Log}}(t_i)$ & $\mu\Delta$ & 0 & $i-1$ & $Y_0$ & 1 \\
					\hline
			\end{tabular}}
	\end{center}}
\end{table}  

\subsubsection{Conditional mean and variance biases}
\label{subsec:4:0:2}

Corollary \ref{Lemma_cond_1} implies that all methods yield conditional means and variances different from the true values. In the following, we study the introduced relative mean and variance biases defined by
\begin{eqnarray}
\label{cond_Bias_mean}\text{rBias}_{\Delta,t_i,Y_0}(\mathbb{E}[\widetilde{Y}])&:=&\frac{\mathbb{E}[\widetilde{Y}(t_i)|Y_0]-\mathbb{E}[Y(t_i)|Y_0]}{\mathbb{E}[Y(t_i)|Y_0]}, \\
\label{cond_Bias_variance}\text{rBias}_{\Delta,t_i,Y_0}(\textrm{Var}(\widetilde{Y}))&:=&\frac{\textrm{Var}(\widetilde{Y}(t_i)|Y_0)-\textrm{Var}(Y(t_i)|Y_0)}{\textrm{Var}(Y(t_i)|Y_0)}, 
\end{eqnarray}
for each considered numerical method. These biases depend
on the time step $\Delta$, the time $t_i$, the initial condition $Y_0$ and the parameters of the model. While the biases in the conditional means of the ODE methods depend on $\sigma$, that of the remaining methods are independent of $\sigma$. The biases in the conditional variance depend on all model parameters. 

\begin{figure}
	\centering
	\includegraphics[width=0.9\textwidth]{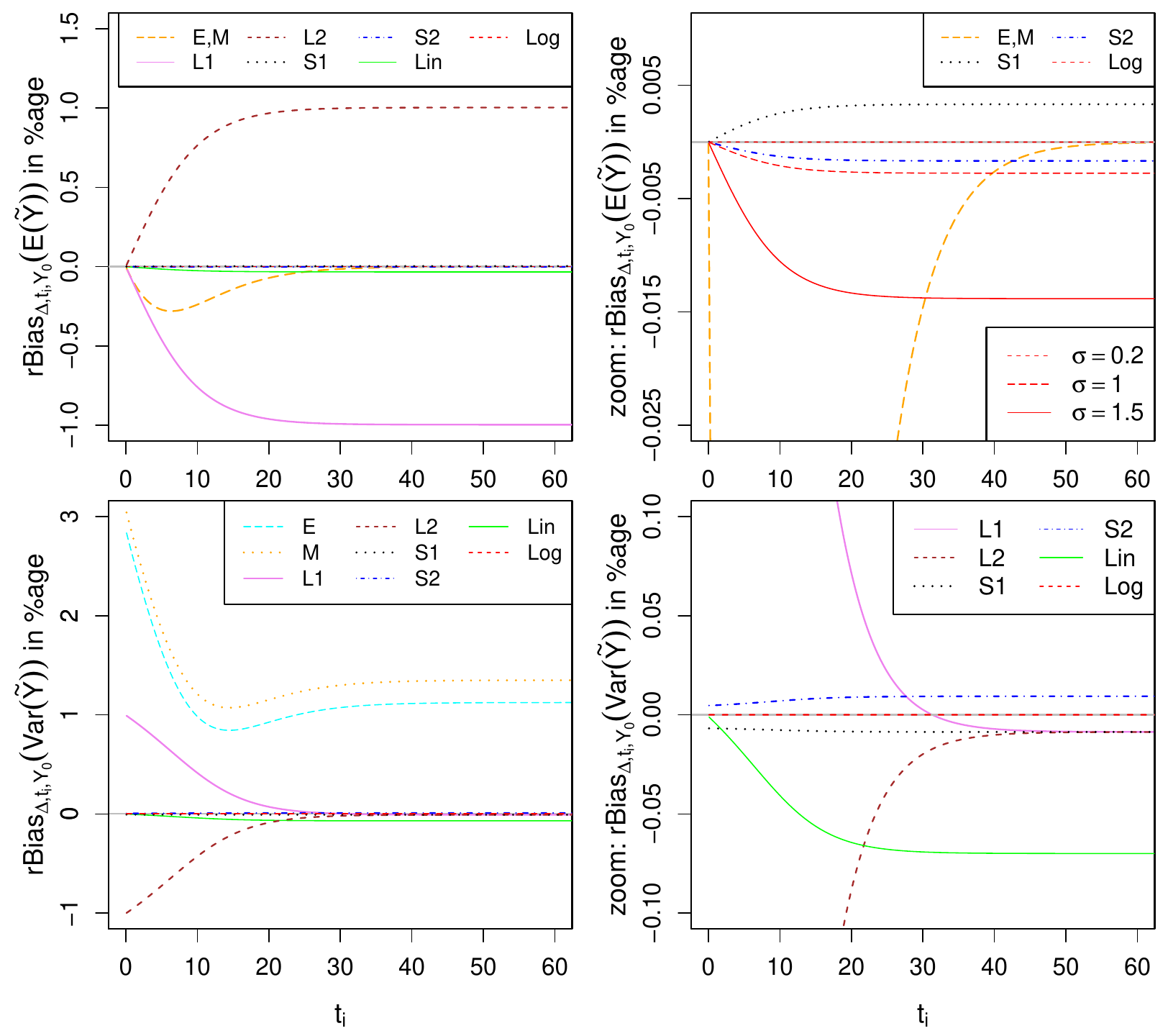}\\
	\caption{Relative conditional mean bias \eqref{cond_Bias_mean} (top left panel for $\textrm{E, M, L1, L2, S1, S2, Lin, Log}$ and a zoom in the top right panel for $\textrm{E, M, S1, S2, Log}$) and conditional variance bias \eqref{cond_Bias_variance} (bottom left panel for $\textrm{E, M, L1, L2, S1, S2, Lin, Log}$ and a zoom in the bottom right panel for $\textrm{L1, L2, S1, S2, Lin, Log}$) in percentage as a function of the time $t_i$, for $Y_0=10$, $\Delta=0.1$, $\mu=1$, $\tau=5$ and $\sigma=0.2$. In the top right panel also $\sigma=1$ and $\sigma=1.5$ are considered.
	}
	\label{Bias_cond_t_y0}
\end{figure}

In the top left panel of Figure \ref{Bias_cond_t_y0}, we report the relative mean bias \eqref{cond_Bias_mean}
in percentage as a function of $t_i$, for $Y_0=10$, $\Delta=0.1$, $\mu=1$,  $\tau=5$ and $\sigma=0.2$. The relative mean biases (in absolute value) introduced by the Strang splitting schemes are significantly smaller than those of the Lie-Trotter splitting schemes and close to $0$ for all $t_i$ under consideration, with the second Strang scheme performing slightly better than the first one (see the top right panel where we provide a zoom). Moreover, the piecewise linear method performs better than the Lie-Trotter methods, but worse than the Strang schemes. For the chosen value of $\sigma$, the log-ODE method outperforms the Strang methods and produces a bias even closer to $0$ for all times $t_i$. However, this fact changes when $\sigma$ is increased, as shown in the top right panel where we also consider $\sigma=1$ and $\sigma=1.5$. In particular, due to the dependence of the mean of the ODE schemes on $\sigma$, they may perform worse than all other methods in terms of preserving the mean when $\sigma$ increases. Furthermore, it can be observed that in the non-stationary initial part, the Strang and ODE methods clearly outperform the Euler-Maruyama and Milstein schemes. This changes with increasing time. In particular, the  relative mean bias of the Euler-Maruyama and Milstein schemes approaches $0$, suggesting an asymptotically unbiased mean (see Subsection \ref{subsec:4:1}).

\begin{figure}
	\centering
	\includegraphics[width=0.9\textwidth]{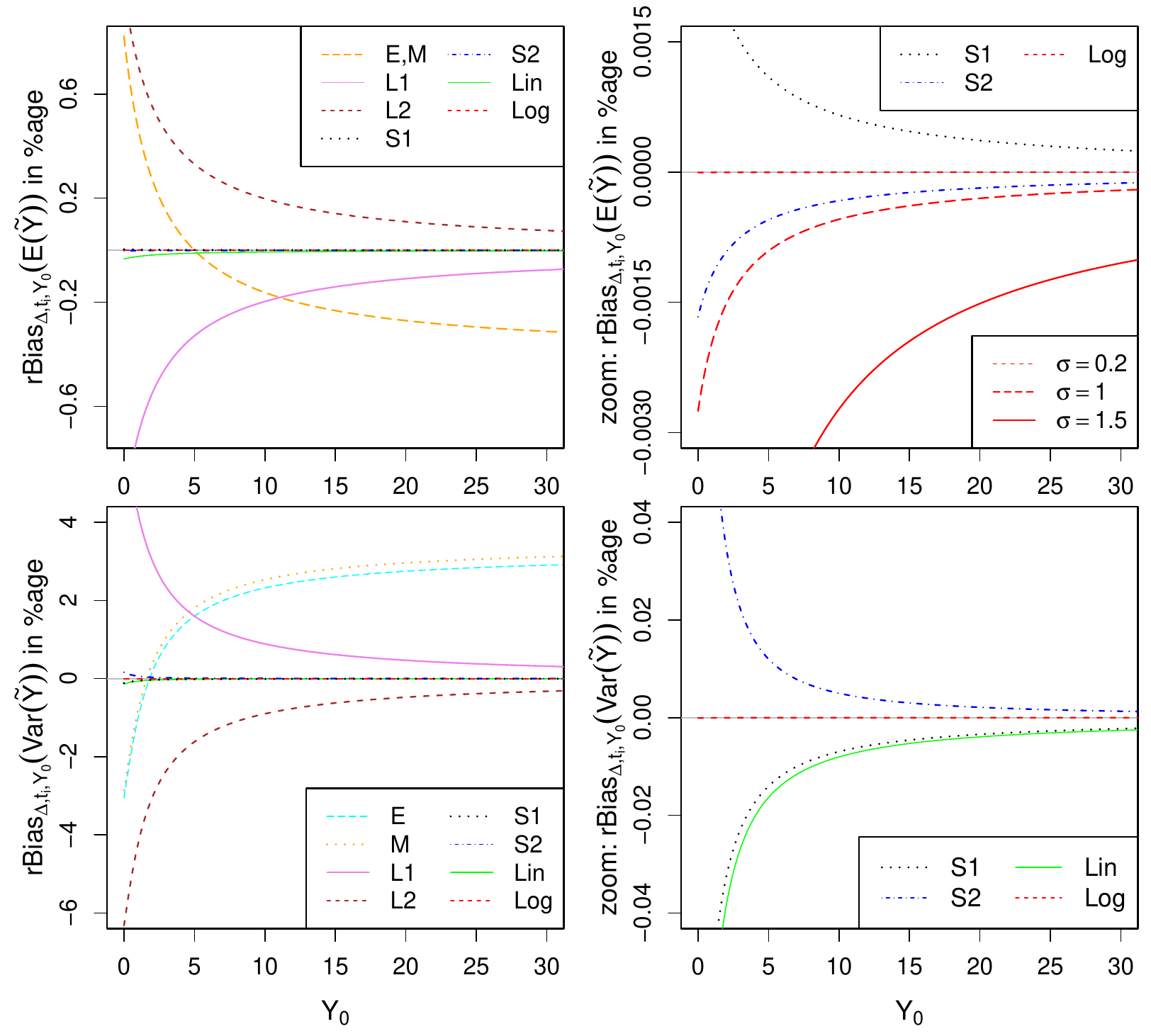}\\
	\caption{Relative conditional mean bias \eqref{cond_Bias_mean} (top left panel for $\textrm{E, M, L1, L2, S1, S2, Lin, Log}$ and a zoom in the top right panel for $\textrm{S1, S2, Log}$) and conditional variance bias \eqref{cond_Bias_variance} (bottom left panel for $\textrm{E, M, L1, L2, S1, S2, Lin, Log}$ and a zoom in the bottom right panel for $\textrm{S1, S2, Lin, Log}$) in percentage as a function of the initial value $Y_0$, for $t_i=2$, $\Delta=0.1$, $\mu=1$, $\tau=5$ and $\sigma=0.2$. In the top right panel also $\sigma=1$ and $\sigma=1.5$ are considered.}
	\label{Bias_cond_delta}
\end{figure}

In the bottom left panel of Figure \ref{Bias_cond_t_y0}, we report the conditional variance biases \eqref{cond_Bias_variance}
in percentage as a function of $t_i$ for the same values of $Y_0$, $\Delta$, $\mu$, $\tau$ and $\sigma=0.2$. All four splitting schemes and both ODE schemes yield better approximations of the conditional variance than the Euler-Maruyama and Milstein schemes for all $t_i$ under consideration. The log-ODE method yields again a bias  close to $0$ from the beginning, outperforming all other methods. This is also the case when $\sigma$ is increased (figures not shown). Except for $t_i$ very small, the Strang schemes outperform the piecewise linear method, and also yield biases close to $0$ from the beginning. Moreover, the relative variance biases (in absolute value) of the Lie-Trotter splitting schemes decrease in time and seem to coincide asymptotically with that of the first Strang scheme (see Subsection \ref{subsec:4:1}), as it can be observed in the bottom right panel of Figure \ref{Bias_cond_t_y0}. Similar results are obtained for other parameter values, time steps and initial conditions.

In Figure \ref{Bias_cond_delta}, we report the relative biases of the conditional mean \eqref{cond_Bias_mean} (top panels) and variance \eqref{cond_Bias_variance} (bottom panels) in percentage as a function of the initial value $Y_0$, for $t_i=2$ and the same parameters as before. All methods introduce larger biases for very small values of $Y_0$. This may be explained by the fact that reproducing the features of the process near a boundary, i.e., near~$0$, is more difficult. For $\sigma=0.2$, the log-ODE method outperforms the other methods, yielding relative biases close to $0$ for any considered choice of the initial condition, not being strongly influenced by it. Similar to before, this changes when $\sigma$ is increased, as illustrated in the top right panel where we also consider $\sigma=1$ and $\sigma=1.5$. The Strang methods (whose mean bias does not depend on~$\sigma$) do then introduce the smallest bias in the conditional mean. In general, the performance of the splitting and ODE schemes improves as $Y_0$ increases, while the Euler-Maruyama and Milstein schemes perform worse for large values of $Y_0$. This is in agreement with the fact that $Y_0$ enters into the conditional means of the splitting and ODE schemes in the same way as in the true quantity, as evident when comparing the expressions reported in Remark~\ref{rem_2} with the true conditional mean \eqref{cond_mean}. In particular, the conditional mean biases (not the relative ones) of the splitting and ODE schemes do not depend on $Y_0$, while those of the Euler-Maruyama and Milstein schemes do. Furthermore, the conditional variance biases introduced by the splitting and ODE schemes depend linearly on $Y_0$, while those of the Euler-Maruyama and Milstein schemes depend quadratically on $Y_0$. If $Y_0$ is close to the asymptotic mean $\mu\tau$, here $5$, the relative mean bias of the Euler-Maruyama and Milstein schemes is almost $0$ (top left panel), in agreement with the fact that they have an asymptotically unbiased mean (see Subsection~\ref{subsec:4:1}). 

\subsection{Investigation of the asymptotic moments}
\label{subsec:4:1}

We now investigate the asymptotic mean and variance of the numerical solutions, i.e.,
\begin{equation*}
\mathbb{E}[\widetilde{Y}_\infty]:=\lim\limits_{i \to \infty} \mathbb{E}[\widetilde{Y}(t_i)|Y_0], \quad \textrm{Var}(\widetilde{Y}_\infty):=\lim\limits_{i \to \infty} \textrm{Var}(\widetilde{Y}(t_i)|Y_0), \quad t_i=i\Delta,
\end{equation*}
comparing them with the true quantities \eqref{statE} and \eqref{statV}, respectively.

\subsubsection{Closed-form expressions for the asymptotic means and variances}
\label{subsec:4:1:1}

In Proposition \ref{Prop_asym_1}, we provide closed-form expressions of the asymptotic mean and variance of the random variable $Z_i$ introduced in Proposition \ref{Prop_cond_1}. As before, these relations allow for a straightforward derivation of the corresponding results for the numerical schemes of interest, including necessary conditions that guarantee the existence of the asymptotic quantities.
\begin{proposition}\label{Prop_asym_1}
	Let the random variable $Z_i$ be defined as in Proposition \ref{Prop_cond_1}. If $|\mu_x|<1$, the asymptotic mean of $Z_i$ is given by
	\begin{equation}\label{AEZi}
	\mathbb{E}[Z_\infty]:=\lim\limits_{i \to \infty} \mathbb{E}[Z_i|Z_0]=c_1\mu_h\frac{\mu_x}{1-\mu_x}+c_1 W_0 \mu_h+c_2.
	\end{equation}
	If, in addition, $r \in (0,1)$, the asymptotic variance of $Z_i$ is given by
	\begin{equation}
	\label{AVZi}\textrm{Var}(Z_\infty):=\lim\limits_{i \to \infty} \textrm{Var}(Z_i|Z_0)=c_1^2\left(\frac{r_h(\mu_x-1)^2+2\mu_h\mu_x p(1-\mu_x)-(1-r)\mu_h^2}{(\mu_x-1)^2(1-r)}\right).
	\end{equation}
\end{proposition}\noindent
The proof of Proposition \ref{Prop_asym_1} is given in Appendix \ref{appC}. 

Based on Proposition \ref{Prop_asym_1}, we derive the asymptotic moments of the considered numerical schemes. 
\begin{corollary}\label{Lemma1}
	Let $\widetilde{Y}(t_i)$ be the numerical solutions defined through \eqref{EM}, \eqref{M}, \eqref{SP1}-\eqref{SP4}, \eqref{Lin} and \eqref{Log}, respectively. The asymptotic means and variances of the Euler-Maruyama and Milstein schemes are given by
	\begin{eqnarray}
	\label{lim_EM_1} &&\text{If} \ \left|1-\frac{\Delta}{\tau}\right|<1, \quad \mathbb{E}[\widetilde{Y}^{\textrm{E}}_\infty]=\mathbb{E}[\widetilde{Y}^{\textrm{M}}_\infty]=\mu\tau, \\
	\label{limVarEM}  \hspace{-1cm} &&\text{If} \ \left|1-\frac{\Delta}{\tau}\right|<1 \ \text{and} \ \Delta<2\tau-\sigma^2\tau^2, \ \textrm{Var}(\widetilde Y_\infty^\textrm{E})=\frac{(\mu\tau)^2}{\frac{2}{\sigma^2\tau}-1-\frac{\Delta}{\sigma^2\tau^2}}, \\
	\label{limVarM} \hspace{-1cm} &&\text{If} \ \left|1-\frac{\Delta}{\tau}\right|<1 \ \text{and} \ \Delta<\frac{2\tau-\sigma^2\tau^2}{\frac{\sigma^4\tau^2}{2}+1}, \ \textrm{Var}(\widetilde Y_\infty^\textrm{M})=\frac{(\mu\tau)^2(1+\frac{\sigma^2\Delta}{2})}{\frac{2}{\sigma^2\tau}-1-\frac{\Delta}{\sigma^2\tau^2}-\frac{\sigma^2\Delta}{2}}.
	\end{eqnarray}
	The asymptotic means and variances of the splitting schemes are given by
	\begin{eqnarray}
	\label{lim_E_1}  && \mathbb{E}[\widetilde{Y}^\textrm{L1}_\infty]={\mu}\tau \left(\frac{\Delta/\tau}{e^{\Delta/\tau}-1} \right),\\
	\label{equiv} &&
	\mathbb{E}[\widetilde Y_\infty^\textrm{L2}]=\mathbb{E}[\widetilde Y_\infty^\textrm{L1}]+\mu\Delta=\mathbb{E}[\widetilde Y_\infty^\textrm{L1}] e^{\Delta/\tau},\\
	\label{equiv2} &&
	\mathbb{E}[\widetilde Y_\infty^\textrm{S1}]=\mathbb{E}[\widetilde Y_\infty^\textrm{L1}]+\frac{\mu\Delta}{2}=\frac{1}{2}\mathbb{E}[\widetilde Y_\infty^\textrm{L1}] (1+e^{\Delta/\tau}),\\
	\label{lim_E_4} && \mathbb{E}[\widetilde{Y}^\textrm{S2}_\infty]=\mathbb{E}[\widetilde Y_\infty^\textrm{L1}]e^{\Delta/2\tau}, \\
	\label{limVar}  \hspace{-1cm} &&\text{If} \ \sigma^2\tau<2, \  \textrm{Var}(\widetilde Y_\infty^\textrm{L1})=\textrm{Var}(\widetilde Y_\infty^\textrm{L2})=\textrm{Var}(\widetilde Y_\infty^\textrm{S1})
	=\mathbb{E}[\widetilde Y_\infty^\textrm{L1}]^2 \frac{e^{2\Delta/\tau}(e^{\Delta\sigma^2}-1)}{e^{2\Delta/\tau}-e^{\Delta\sigma^2}}, \\
	\label{limVar4} \hspace{-1cm}&&\text{If}  \ \sigma^2\tau<2, \ \textrm{Var}(\widetilde Y_\infty^\textrm{S2})
	=\mathbb{E}[\widetilde Y_\infty^\textrm{L1}]^2 \frac{e^{\Delta/\tau}(e^{\Delta\sigma^2/2}-1)(e^{2\Delta/\tau}+e^{\Delta\sigma^2/2})}{e^{2\Delta/\tau}-e^{\Delta\sigma^2}}.
	\end{eqnarray}
	The asymptotic means and variances of the ODE schemes are given by
	\begin{eqnarray}
	\label{lim_E_Lin} \hspace{-2cm} && \mathbb{E}[\widetilde{Y}^\textrm{Lin}_\infty]=\mathbb{E}[\widetilde Y_\infty^\textrm{L1}]e^{\Delta/\tau}L_{\Delta,\tau,\sigma}, \\
	\label{lim_E_Log} \hspace{-2cm} && \mathbb{E}[\widetilde{Y}^\textrm{Log}_\infty]=\mathbb{E}[\widetilde Y_\infty^\textrm{L1}]e^{\Delta/\tau}L_{\Delta,\tau,\sigma}\left( 1+\sigma^2 \frac{\Delta}{12} \right), \\
	\label{limVarLin} \hspace{-2cm}&&\text{If} \  \sigma^2\tau<2, \ \textrm{Var}(\widetilde Y_\infty^\textrm{Lin})
	=\mathbb{E}[\widetilde Y_\infty^\textrm{L1}]^2 \frac{e^{2\Delta/\tau}\left( 2L\widetilde{L}(e^{\Delta/\tau}-1)+\bar{L}(e^{\Delta/\tau}-1)^2+L^2(e^{\Delta\sigma^2}-e^{2\Delta/\tau}) \right)}{e^{2\Delta/\tau}-e^{\Delta\sigma^2}}, \\
	\label{limVarLog} \hspace{-2cm}&&\text{If} \  \sigma^2\tau<2, \ \textrm{Var}(\widetilde Y_\infty^\textrm{Log})
	=\mathbb{E}[\widetilde Y_\infty^\textrm{L1}]^2 \frac{e^{2\Delta/\tau}\left( 2K\widetilde{K}(e^{\Delta/\tau}-1)+\bar{K}(e^{\Delta/\tau}-1)^2+K^2(e^{\Delta\sigma^2}-e^{2\Delta/\tau}) \right)}{e^{2\Delta/\tau}-e^{\Delta\sigma^2}},
	\end{eqnarray}
	where $L\equiv L_{\Delta,\tau,\sigma}$, $\bar{L}\equiv \bar{L}_{\Delta,\tau,\sigma}$, $\widetilde{L}\equiv \widetilde{L}_{\Delta,\tau,\sigma}$, $K\equiv K_{\Delta,\tau,\sigma}$, $\bar{K}\equiv \bar{K}_{\Delta,\tau,\sigma}$, $\widetilde{K}\equiv \widetilde{K}_{\Delta,\tau,\sigma}$ are as in \eqref{L_delta_tau_sigma} and \eqref{barL_delta_tau_sigma}-\eqref{tildeK_delta_tau_sigma}, respectively. 
\end{corollary}
\begin{proof}
	The results and their required conditions follow directly from Proposition \ref{Prop_asym_1}, using the corresponding values reported in Table \ref{table1} and simplifying the resulting expressions.
\end{proof}

Remarkably, the splitting and ODE methods do not require extra conditions for the existence of the asymptotic mean, but, between the two, only the splitting schemes have asymptotic means independent on $\sigma$, as it is the case for the IGBM. Moreover, the condition guaranteeing the existence of the asymptotic variance of the splitting and ODE schemes is the same as that of the true process, i.e., $\sigma^2\tau<2$. In contrast, the Euler-Maruyama and the Milstein schemes rely on extra conditions that do not depend on the features of the model. If $|1-\Delta/\tau|<1$, the Euler-Maruyama and the Milstein schemes have unbiased asymptotic means. Regarding the asymptotic variance, the condition for the Milstein scheme in \eqref{limVarM} is more restrictive than that for the Euler-Maruyama method in \eqref{limVarEM}, agreeing with similar results in the literature \cite{Buckwar2011}. The asymptotic variances of the Lie-Trotter schemes and the first Strang scheme coincide, as previously hypothesised looking at Figure \ref{Bias_cond_t_y0}.

\vspace{-0.3cm}
If $\mu=0$, the results for the Euler-Maruyama and Milstein methods in Corollary \ref{Lemma1} are in agreement with those available in the linear stochastic stability literature for the GBM \cite{Higham2000,Saito1996}. In particular, the conditions required in \eqref{limVarEM} and \eqref{limVarM} are the same as those guaranteeing their mean-square stability. On the contrary, Corollary \ref{Lemma1} implies that the splitting \eqref{SP1}-\eqref{SP4} and  ODE \eqref{Lin}, \eqref{Log} schemes are asymptotically first and second moment stable without needing extra~conditions.

\begin{figure}
	\centering
	\includegraphics[width=0.9\textwidth]{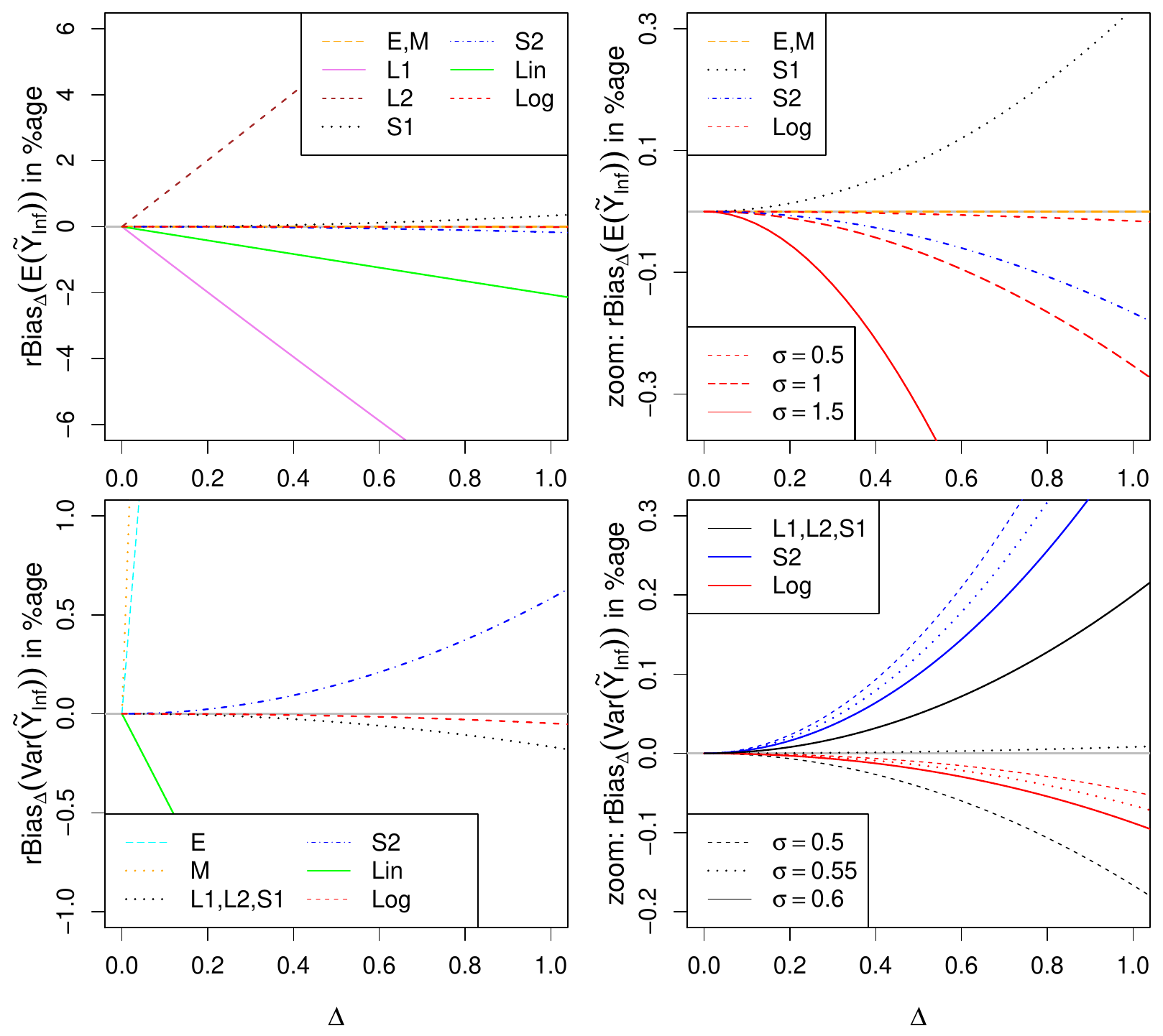}\\
	\caption{Relative asymptotic mean bias \eqref{asym_bias_mean} (top left panel for $\textrm{E, M, L1, L2, S1, S2, Lin, Log}$ and a zoom in the top right panel for $\textrm{E, M, S1, S2, Log}$) and asymptotic variance bias \eqref{asym_bias_variance} (bottom left panel for $\textrm{E, M, L1, L2, S1, S2, Lin, Log}$ and a zoom in the bottom right panel for $\textrm{L1, L2, S1, S2, Log}$) in percentage as a function of $\Delta$, for $\tau=5$ and $\sigma=0.5$. In the top right and bottom right panels also $\sigma=1$, $\sigma=1.5$ and $\sigma=0.55$, $\sigma=0.6$ are considered, respectively.}
	\label{Bias1_2_3}\vspace{-0.5cm}
\end{figure}

\subsubsection{Asymptotic mean and variance biases}
\label{subsec:4:1:1:new}

Corollary \ref{Lemma1} implies that the derived schemes introduce asymptotic mean and variance biases. In the following, we analyse the resulting asymptotic relative biases 
\begin{eqnarray}
\label{asym_bias_mean}\text{rBias}_\Delta\left(\mathbb{E}[\widetilde{Y}_\infty]\right)&:=&\frac{\mathbb{E}[\widetilde{Y}_\infty]-\mathbb{E}[Y_\infty]}{\mathbb{E}[Y_\infty]}, \\
\label{asym_bias_variance}\text{rBias}_\Delta\left(\textrm{Var}(\widetilde{Y}_\infty)\right)&:=&\frac{\textrm{Var}(\widetilde{Y}_\infty)-\textrm{Var}(Y_\infty)}{\textrm{Var}(Y_\infty)}, 
\end{eqnarray}
with respect to the true quantities \eqref{statE} and \eqref{statV}, for each considered numerical method. These biases depend on the time step $\Delta$ and on the model parameters. All relative asymptotic biases do not depend on $\mu$. In particular, except for the ODE methods, the asymptotic mean biases depend only on the ratio $\Delta/\tau$ and their asymptotic variance biases depend on both $\Delta/\tau$ and $\Delta\sigma^2$. 
As expected, all biases vanish as $\Delta\to0$, provided that the conditions of Corollary \ref{Lemma1} are satisfied.

\begin{figure}
	\centering
	\includegraphics[width=0.9\textwidth]{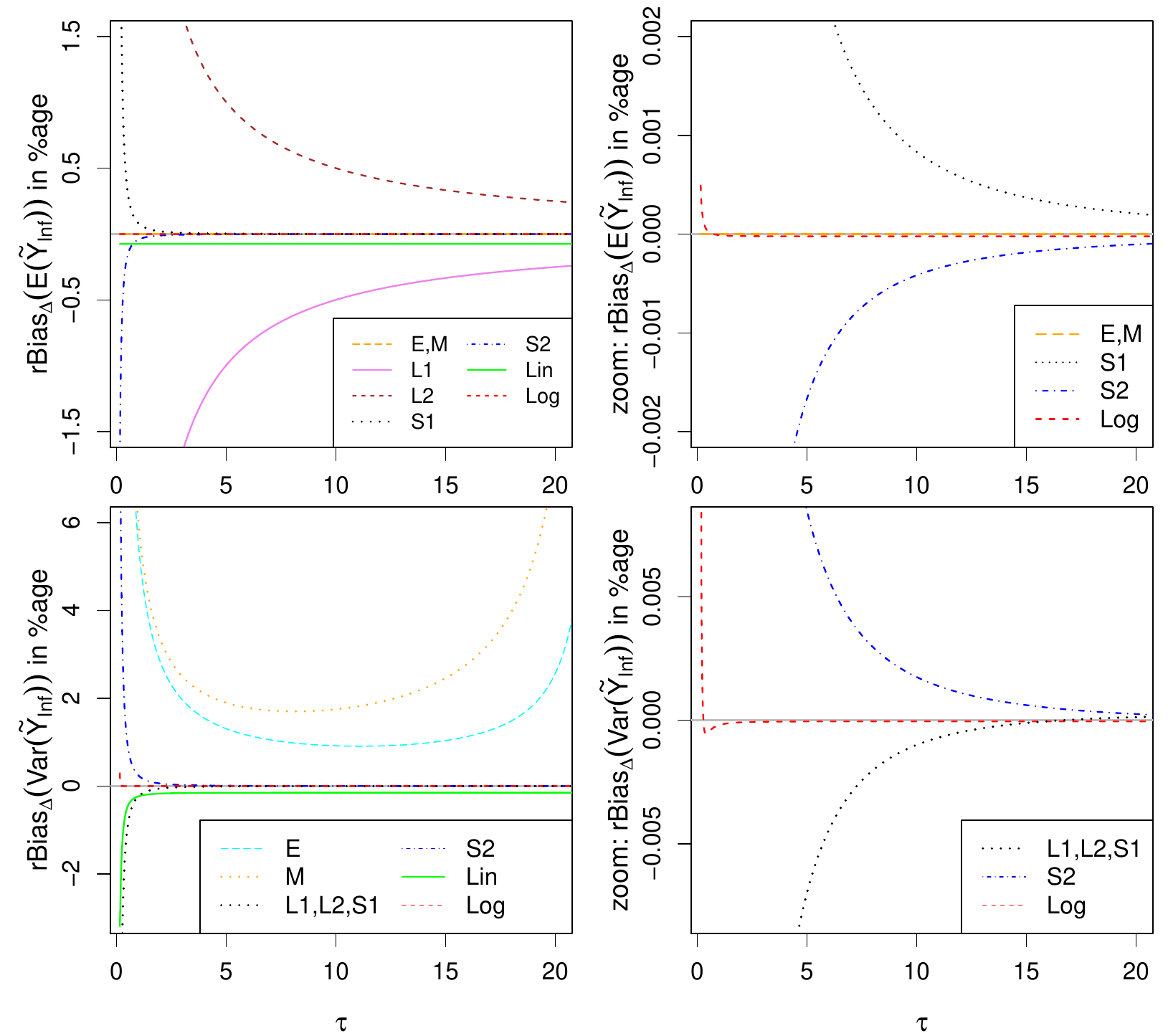}\\
	\caption{Relative asymptotic mean bias \eqref{asym_bias_mean} (top left panel for $\textrm{E, M, L1, L2, S1, S2, Lin, Log}$ and a zoom in the top right panel for $\textrm{E, M, S1, S2, Log}$) and asymptotic variance bias \eqref{asym_bias_variance} (bottom left panel for $\textrm{E, M, L1, L2, S1, S2, Lin, Log}$ and a zoom in the bottom right panel for $\textrm{L1, L2, S1, S2, Log}$) in percentage as a function of $\tau$, for $\Delta=0.1$ and $\sigma=0.3$.}
	\label{Bias_tau}
\end{figure}

In the top left panel of Figure \ref{Bias1_2_3}, we report the relative biases of the asymptotic mean \eqref{asym_bias_mean} in percentage as a function of the time step $\Delta$, for $\tau=5$ and $\sigma=0.5$. Only the asymptotic mean of the Euler-Maruyama and Milstein methods is unbiased. Moreover, independent of the choice of the model parameters and for any time step $\Delta>0$, the Strang schemes yield significantly smaller asymptotic mean biases (in absolute value) than the Lie-Trotter schemes, in agreement with the results reported in the previous section. Moreover, the mean bias (in absolute value) of the second Strang scheme is slightly smaller than that of the first Strang scheme as highlighted in the top right panel, where we provide a zoom. In addition, the log-ODE method introduces a smaller bias in the asymptotic mean than the piecewise linear method. This does not change when considering other values for $\tau$ and $\sigma$, see the top panels of Figure \ref{Bias_tau} and Figure \ref{Bias_asym_sigma} where we fix $\Delta=0.1$ and consider \eqref{asym_bias_mean} as a function of $\tau$ and $\sigma$, respectively. Furthermore, for small values of $\sigma$, the log-ODE method performs better than the Strang schemes in terms of preserving the asymptotic mean. However, this changes when $\sigma$ is increased, see the top right panel of Figure \ref{Bias1_2_3} and the top panels of Figure \ref{Bias_asym_sigma}.

\begin{figure}[H]
	\centering
	\includegraphics[width=0.9\textwidth]{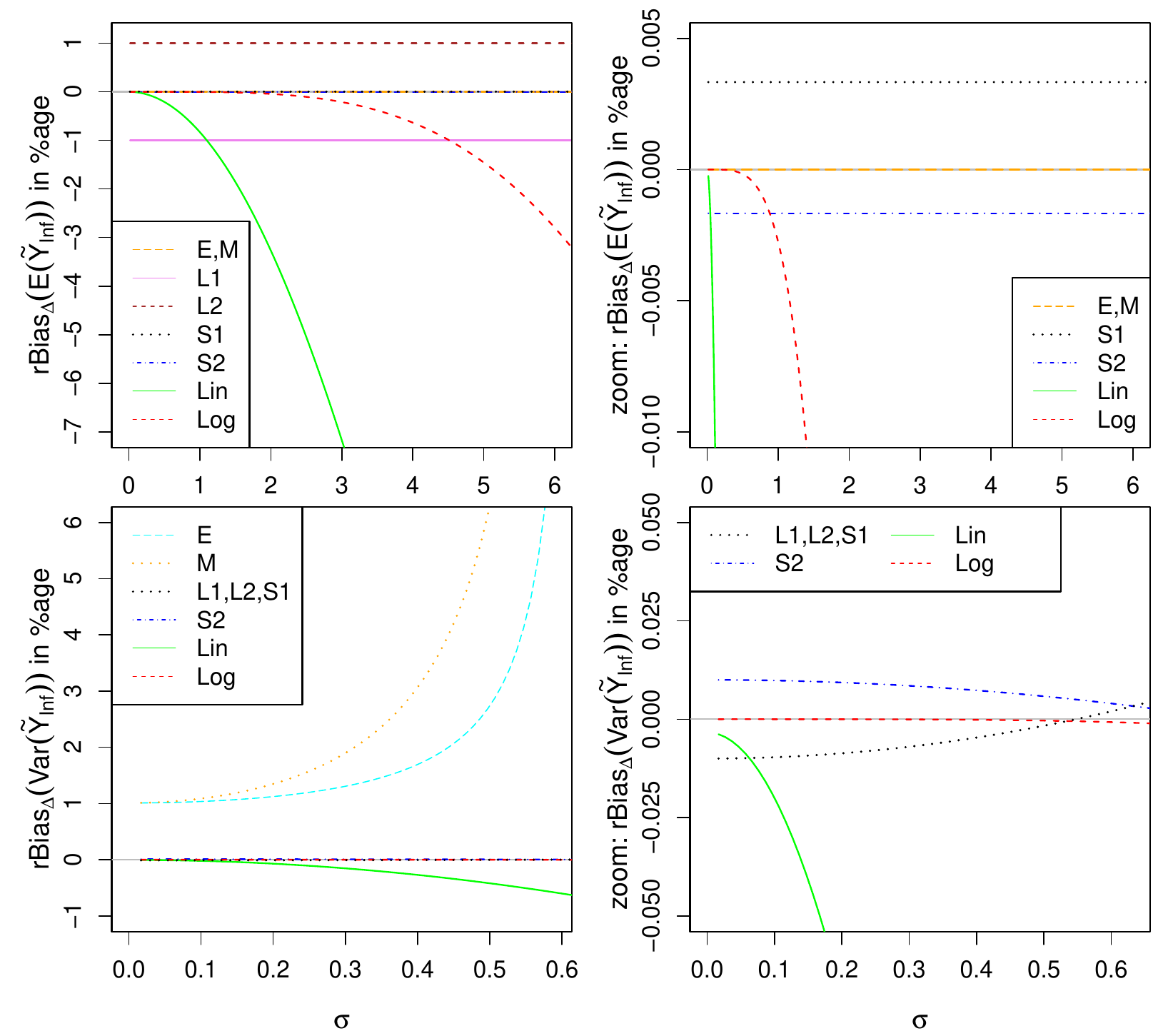}\\
	\caption{Relative asymptotic mean bias \eqref{asym_bias_mean} (top left panel for $\textrm{E, M, L1, L2, S1, S2, Lin, Log}$ and a zoom in the top right panel for $\textrm{E, M, S1, S2, Lin, Log}$) and asymptotic variance bias \eqref{asym_bias_variance} (bottom left panel for $\textrm{E, M, L1, L2, S1, S2, Lin, Log}$ and a zoom in the bottom right panel for $\textrm{L1, L2, S1, S2, Lin, Log}$) in percentage as a function of $\sigma$, for $\tau=5$ and $\Delta=0.1$.}
	\label{Bias_asym_sigma}\vspace{-0.2cm}
\end{figure}

In the bottom left panel of Figure \ref{Bias1_2_3}, we report the relative biases of the asymptotic variance \eqref{asym_bias_variance} in percentage as a function of the time step $\Delta$, for $\tau=5$ and $\sigma=0.5$ fulfilling the conditions of Corollary \ref{Lemma1}. Note that, the Milstein scheme introduces a larger bias in the variance than the Euler-Maruyama method. Moreover, all splitting schemes yield significantly smaller asymptotic variance biases (in absolute value) than the Euler-Maruyama, Milstein and piecewise linear methods. The log-ODE method, however, outperforms the splitting methods. This fact holds true also for other values of $\tau$ and $\sigma$, see the bottom panels of Figure~\ref{Bias_tau} and Figure \ref{Bias_asym_sigma}, where we fix $\Delta=0.1$ and plot \eqref{asym_bias_variance} as a function of $\tau$ and $\sigma$, respectively. However, there exist combinations of $\tau$ and $\sigma$ for which the condition $\sigma^2\tau<2$ is satisfied and the first Strang and Lie-Trotter methods outperform the log-ODE method in terms of asymptotic variance preservation. This is illustrated in Figure \ref{Bias_heat}, where we provide a heatmap of
\vspace{-0.3cm}\begin{equation}\label{r}
\text{ratioBias}:=\frac{\left|\text{rBias}_\Delta\left(\textrm{Var}(\widetilde{Y}^{\textrm{L1,L2,S1}}_\infty)\right)\right|}{\left|\text{rBias}_\Delta\left(\textrm{Var}(\widetilde{Y}^{\textrm{Log}}_\infty)\right)\right|},
\end{equation}
for different values of $\tau$ and $\sigma$. In particular, the region within the white lines corresponds to combinations of $\tau$ and $\sigma$ for which this ratio is smaller than $1$, i.e., for which the first Strang and Lie-Trotter methods introduce a smaller relative asymptotic variance bias (in absolute value) than the log-ODE method. This can be also observed in the bottom right panel of Figure \ref{Bias1_2_3}, where we compare the log-ODE and splitting methods for $\sigma=0.5,0.55,0.6$, with $\sigma=0.55$ within the region marked by the white lines of Figure \ref{Bias_heat}.

\begin{figure}
	\centering
	\includegraphics[width=0.6\textwidth]{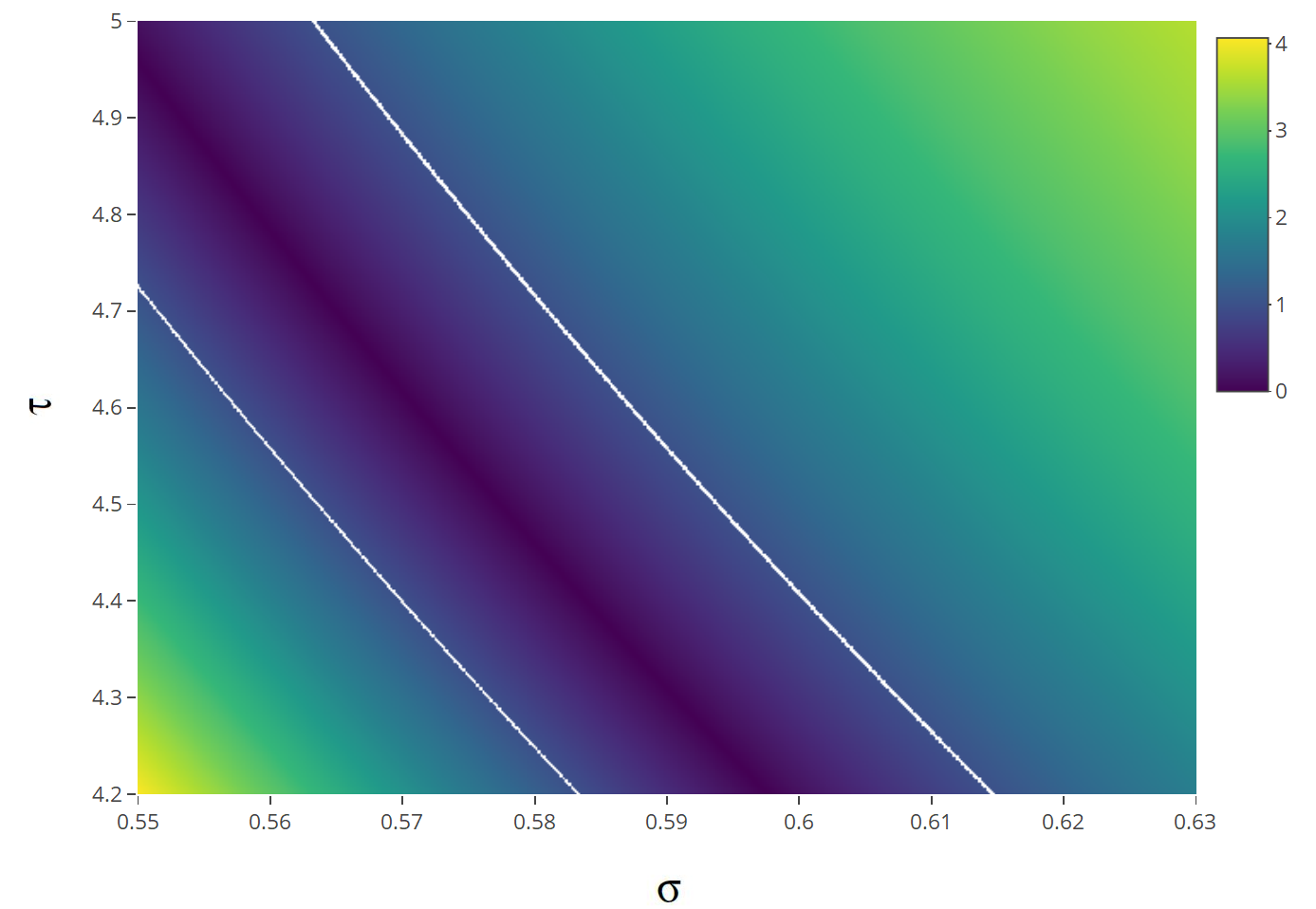}\\
	\caption{Ratio \eqref{r} of the relative asymptotic variance biases of the $\textrm{L1, L2, S1}$ and $\textrm{Log}$ schemes, for different values of $\tau$ and $\sigma$. For the region within the white lines this ratio is smaller than $1$.}
	\label{Bias_heat} \vspace{-0.3cm}
\end{figure}

Note also that all biases increase when $\tau$ is very small (cf. Figure \ref{Bias_tau}). This is because all biases depend on the ratio $\Delta/\tau$, requiring a small value of the time step $\Delta$ to keep this ratio constant when $\tau$ is small. Moreover, for small values of $\tau$ the process gets closer to the boundary, where its behaviour is more difficult to preserve. The ODE methods, however, are less deterred by small values of $\tau$, possibly due to their dependence on $L_{\Delta,\tau,\sigma}$ \eqref{L_delta_tau_sigma}, $\bar{L}_{\Delta,\tau,\sigma}$ \eqref{barL_delta_tau_sigma} and $\widetilde{L}_{\Delta,\tau,\sigma}$ \eqref{tildeL_delta_tau_sigma}. Moreover, the relative biases (in absolute value) of the splitting and ODE methods decrease for large values of $\tau$, while the relative variance biases (in absolute value) of the Euler-Maruyama and Milstein schemes initially decrease and then increase. 

Interestingly, while the relative asymptotic variance biases (in absolute value) of the Euler-Maruyama, Milstein and ODE schemes increase in $\sigma$ (bottom panels of Figure \ref{Bias_asym_sigma}), that of the second Strang method decreases as $\sigma$ increases (bottom right panel), and that of the first Strang and Lie-Trotter schemes first decreases, and then increases again. The latter scenario is related to the fact that there do exist parameter values for which the first Strang and Lie-Trotter methods introduce a smaller relative variance bias (in absolute value) than the log-ODE method, see Figure~\ref{Bias_heat}.

In general, if both $\tau$ and $\sigma$ are large, such that the stationary condition $\sigma^2\tau<2$ is only met tightly, the splitting and log-ODE schemes perform well compared to the Euler-Maruyama, Milstein and piecewise linear method in terms of preserving the asymptotic variance. In particular, even though the Strang and log-ODE schemes perform slightly worse than the Euler-Maruyama and Milstein schemes in terms of the asymptotic mean, they clearly outperform them in terms of the asymptotic variance. For this reason, when, for example, analysing the asymptotic coefficient of variation, i.e., $\textrm{CV}(Y_\infty):=~\sqrt{\textrm{Var}(Y_\infty)}/\mathbb{E}[Y_\infty]$, which is a measure of dispersion that allows to simultaneously study the error impinging on both quantities, the Strang and log-ODE schemes are superior to all other schemes. Moreover, if the stationary condition $\sigma^2\tau<2$ is only met tightly, 
the first Strang scheme outperforms the second one in terms of the CV.

\vspace{-0.2cm}
\subsection{Preservation of the boundary properties}
\label{subsec:4:2}

As discussed in Section \ref{sec2}, the boundary $0$ of the IGBM may be of entrance, unattainable and attracting or exit type, depending on the parameter $\mu$. Corresponding  properties motivated by this classification have been introduced at the end of Section \ref{sec2}. A numerical scheme $\widetilde{Y}(t_i)$ is said to preserve these properties if the following discrete versions are fulfilled: \vspace{-0.1cm}
\begin{itemize}
	\item[$\bullet$] \textit{Discrete unattainable property}: If $\mu \geq 0$, then $\mathbb{P}(\widetilde{Y}(t_i)>0|\widetilde{Y}(t_{i-1})>0)=1$. \vspace{-0.15cm}
	\item[$\bullet$] \textit{Discrete absorbing property}: If $\mu=0$, then $\mathbb{P}(\widetilde{Y}(t_1)=0|Y_0=0)=1$. \vspace{-0.15cm}
	\item[$\bullet$] \textit{Discrete entrance property}: If $\mu>0$, then $\mathbb{P}(\widetilde{Y}(t_1)>0|Y_0=0)=1$. \vspace{-0.15cm}
	\item[$\bullet$] \textit{Discrete exit property}: If $\mu<0$, then $\mathbb{P}(\widetilde{Y}(t_i)<0|\widetilde{Y}(t_{i-1}) \leq 0)=1$.
\end{itemize}\vspace{-0.1cm}
It is well known that the Euler-Maruyama and Milstein schemes may fail in meeting such conditions. For example, the Euler-Maruyama scheme \eqref{EM} does not fulfill the discrete unattainable property for any choice of $\Delta$, since $\xi_{i-1}$ assumes all values in $\mathbb{R}$ with a positive probability \cite{Kahl2008}. Moreover, the Milstein scheme \eqref{M} may not fulfill this property either, if $\widetilde{Y}(t_{i-1})(\sigma^2 +\frac{2}{\tau})-2\mu>0$, unless the time discretisation step $\Delta$ satisfies
\begin{equation}\label{cond_Delta}
\Delta < \frac{2yG'(y)-G(y)}{(G(y)G'(y)-2F(y))G'(y)}=\frac{y}{\sigma^2 y + \frac{2}{\tau}y-2\mu}, \vspace{-0.2cm}
\end{equation}
where $y:=\widetilde{Y}(t_{i-1})$ and $G'(y)$ denotes the derivative of $G$ with respect to $y$ \cite{Kahl2008}. Thus, to guarantee positivity, the time step $\Delta$ would need to be updated in every iteration step. 

Note that, the discrete absorbing and entrance properties are the only properties which are satisfied by the Euler-Maruyama and Milstein schemes, for any time step $\Delta$. In contrast, the ODE methods and the derived splitting schemes preserve the different boundary properties for any choice of time step $\Delta>0$, as shown below. Moreover, their boundary behaviour depends only on the parameter $\mu$, as it is the case for the IGBM. 
\vspace{-0.1cm}\begin{proposition}\label{Theorem2}
	Let $\widetilde{Y}(t_i)$ be the splitting, piecewise linear and log-ODE schemes defined through \eqref{SP1}-\eqref{SP4}, \eqref{Lin} and \eqref{Log}, respectively. They fulfill the discrete unattainable, absorbing, entrance and exit properties for any choice of the time step $\Delta$.
\end{proposition}\vspace{-0.1cm}
The discrete boundary properties can be verified from  \eqref{SP1}-\eqref{SP4}, \eqref{Lin} and \eqref{Log}, using the corresponding assumptions on the parameter $\mu$ and the positivity of the exponential function. A detailed proof of Proposition~\ref{Theorem2} is given in Appendix \ref{appD}.

\vspace{-0.4cm}
\section{Simulation results}
\label{sec5}
\vspace{-0.2cm}

We now illustrate the theoretical results introduced in the previous sections through a series of simulations. First, we represent graphically the mean-square convergence order of the different numerical methods and discuss their required computational effort. Second, we focus on the conditional and asymptotic moments. Third, we compare the ability of the different methods to estimate the stationary density of the process. Finally, we consider the boundary properties, and provide a further investigation of the behaviour of the numerical solutions at the boundary.

\subsection{Mean-square convergence order and computational effort}
\label{sec5:1}

The mean-square convergence order of the different numerical methods can be approximated via the root mean-squared error (RMSE) considered as a function of the time step $\Delta$. In particular, we define
\begin{equation}\label{RMSE}
\text{RMSE}(\Delta):=\left( \frac{1}{n} \sum\limits_{k=1}^{n} \left| Y_k(t_{\text{max}})-\widetilde{Y}_k(t_{\text{max}}) \right|^2 \right)^{1/2}, 
\end{equation}
where ${Y}_k(t_{\text{max}})$ and $\widetilde{Y}_k(t_{\text{max}})$ denote the $k$-th realisation and approximation (obtained under a numerical method using the time step $\Delta$) of the process, respectively, at a fixed time $t_{\text{max}}$.

In the left panel of Figure \ref{Fig_convergence}, we report the RMSEs of the different schemes as a function of the time step $\Delta$ and in $\textrm{log10}$ scale. We use the same parameter setting as in Figure 4.2 in \cite{Foster2020}, i.e., we fix $n=10^5$, $t_{\text{max}}=5$, $Y_0=0.06$, $\mu=0.004$, $\tau=10$ and $\sigma=0.6$. Since the IGBM is not known explicitly, the values ${Y}_k(t_{\text{max}})$ are obtained under the log-ODE method, using the small time step $\Delta=2^{-10}$. The approximated values $\widetilde{Y}_k(t_{\text{max}})$ are produced under the considered numerical methods and for different values of $\Delta$, specifically $\Delta=2^{-l}$, $l=0,\ldots,8$. Note that the ${Y}_k(t_{\text{max}})$ and $\widetilde{Y}_k(t_{\text{max}})$ have to be computed with respect to the same Brownian paths, see \cite{Foster2020} and its supporting code for how to deal with the rescaled space-time L\'{e}vy areas of a Brownian~increment. \newpage

\begin{figure}
	\centering
	\includegraphics[width=0.8\textwidth]{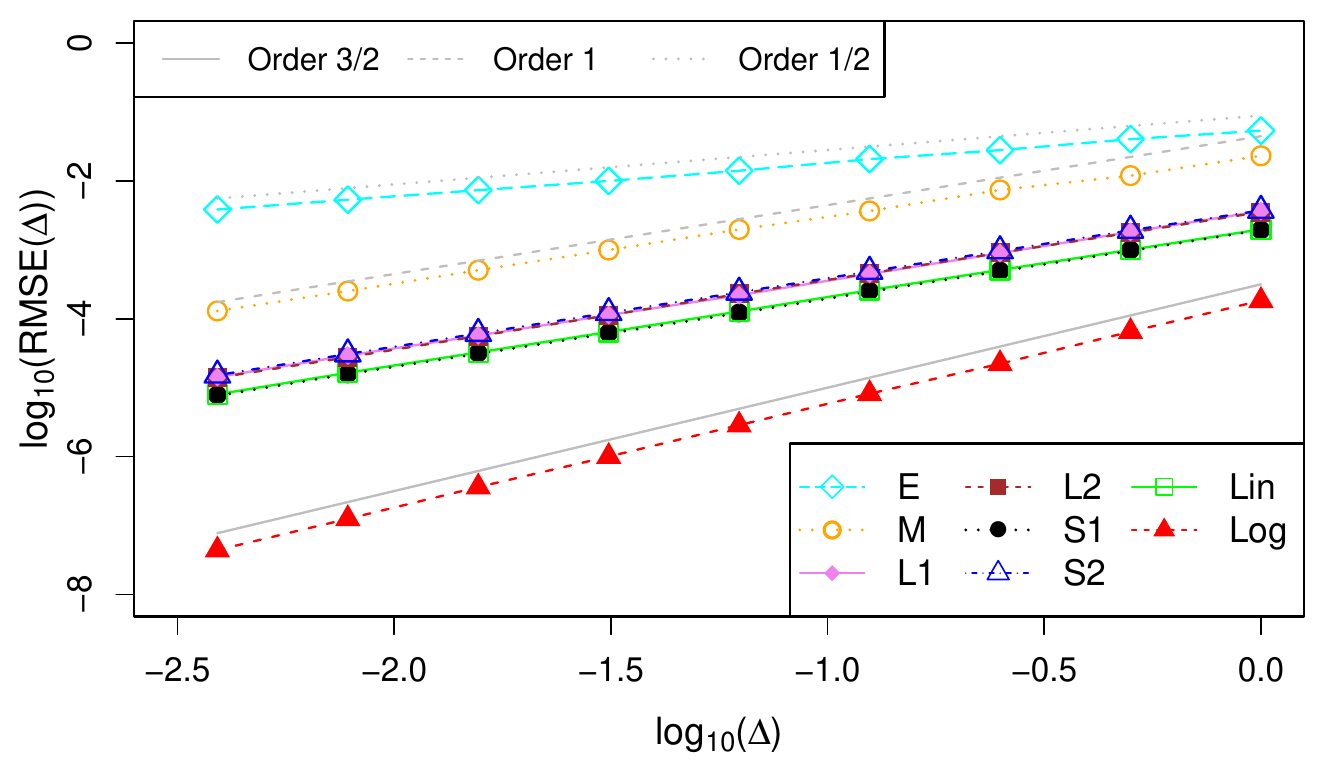}\\\vspace{-0.3cm}
	\caption{RMSE \eqref{RMSE} in $\textrm{log10}$ scale for $\textrm{E, M, L1, L2, S1, S2, Lin, Log}$ as a function of $\Delta$, for $n=10^5$, $t_{\text{max}}=5$, and the same parameter values as used in Figure 4.2 in \cite{Foster2020}, i.e., $Y_0=0.06$, $\mu=0.004$, $\tau=10$ and $\sigma=0.6$.}
	\label{Fig_convergence}
\end{figure}

\begin{table}
	{\footnotesize  \vspace{-0.1cm}
		\caption{Number of operations, function evaluations and random numbers required per iteration, i.e., required to produce $\widetilde{Y}(t_i)$ given $\widetilde{Y}(t_{i-1})$, $\Delta$, $\tau$, $\mu$ and $\sigma$.}
		\label{table_Operations}
		\begin{center} \vspace{-0.5cm}
			\scalebox{1.1}{
				\begin{tabular}{|c!{\vrule width 1pt}c|c|c|c!{\vrule width 1pt}c|}
					\hline 
					Effort & $+,-,\times,/$ & $\sqrt{\cdot}$ & $\exp(\cdot)$ & $N(0,1)$ & $\sum$ \\ \noalign{\hrule height 1pt} 
					$\textrm{E}$ & $10$ & $1$ & $0$ & $1$ & $12$ \\
					$\textrm{M}$ & $15$ & $1$ & $0$ & $1$ & $17$ \\
					$\textrm{L1}$ & $12$ & $1$ & $1$ & $1$ & $15$ \\
					$\textrm{L2}$ & $12$ & $1$ & $1$ & $1$ & $15$ \\
					$\textrm{S1}$ & $14$ & $1$ & $1$ & $1$ & $17$ \\
					$\textrm{S2}$ & $19$ & $1$ & $2$ & $2$ & $24$ \\
					$\textrm{Lin}$ & $15$ & $1$ & $1$ & $1$ & $18$ \\
					$\textrm{Log}$ & $26$ & $2$ & $1$ & $2$ & $31$ \\
					\hline
			\end{tabular}}
	\end{center}} \vspace{-0.6cm}
\end{table}

As expected, we observe a mean-square convergence rate of order $3/2$ for the log-ODE scheme, a rate of order $1$ for the Milstein, piecewise linear and splitting methods, and a rate of order $1/2$ for the Euler-Maruyama discretisation. The log-ODE method yields the smallest RMSEs, and the Euler-Maruyama method produces the largest error estimates.  Among the order $1$ methods we observe differences in their accuracies. The first Strang scheme yields the smallest RMSEs, with error estimates slightly smaller that of the piecewise linear method. Moreover, the RMSEs of the two Lie-Trotter and second Strang methods are almost the same, the second Strang method performing slightly worse than the Lie-Trotter schemes. The Milstein method yields the largest error estimates in the considered class of order $1$ methods.

These results should be considered in relation to the computational effort required by the different schemes to generate a path, see, e.g., \cite{Debrabant2009}. We measure this effort by counting the number of operations, function evaluations and random numbers required per iteration, i.e., required to produce $\widetilde{Y}(t_i)$ given $\widetilde{Y}(t_{i-1})$, $\Delta$, $\tau$, $\mu$ and $\sigma$. This is summarised in Table \ref{table_Operations}. The log-ODE method requires the largest computational effort, and the Euler-Maruyama method the slightest. While the effort required by the two Lie-Trotter schemes is the same, the effort of the second Strang method clearly exceeds that of the first. Moreover, while the second Strang scheme yields almost the same RMSEs as the Lie-Trotter schemes (cf. Figure \ref{Fig_convergence}), it requires a greater effort to produce these errors (cf. Table \ref{table_Operations}). 

\begin{figure}
	\centering
	\includegraphics[width=0.9\textwidth]{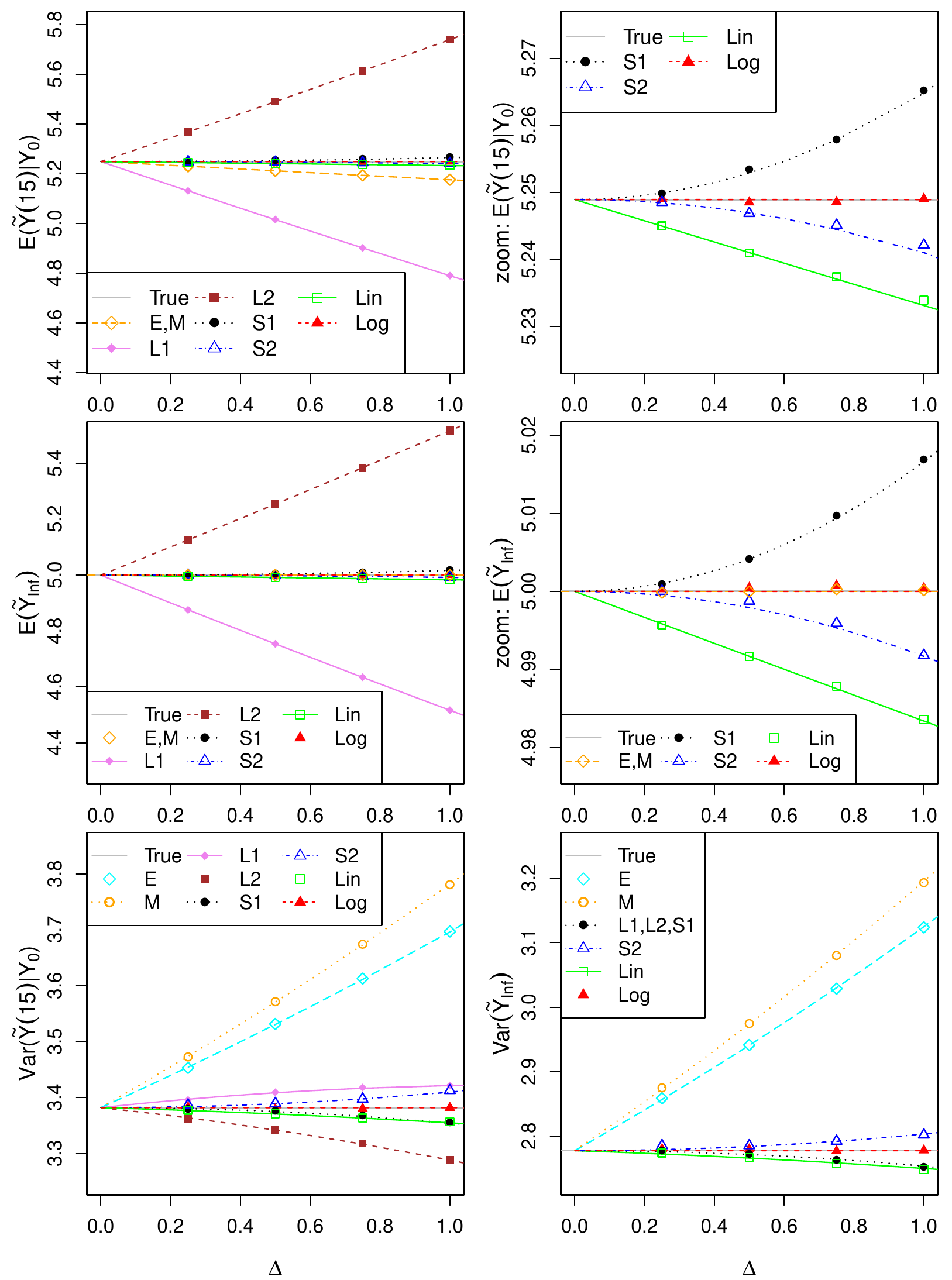}\\
	\caption{Theoretical conditional and asymptotic means and variances of the different numerical methods (lines) as functions of $\Delta$, for $\mu=1$, $\tau=5$, $\sigma=0.2$, $Y_0=10$ and $t_i=15$, and corresponding values obtained via simulations (symbols), for $\Delta=0.25,0.5,0.75,1$. The true conditional and asymptotic mean and variance of the IGBM are represented by the grey horizontal~lines.}
	\label{Fig_moments_cond}
\end{figure}

\subsection{Conditional and asymptotic moments}

Here, we illustrate that the conditional and asymptotic means and variances obtained via numerical simulations are in agreement with the previously derived theoretical expressions. To do so, we define the sample mean $\hat{m}_{t_i}$ and variance $\hat{v}_{t_i}$ as follows
\begin{eqnarray}
\mathbb{E}[Y(t_i)|Y_0]\label{m}&\approx& \mathbb{E}[\widetilde{Y}(t_i)|Y_0] \approx \hat{m}_{t_i}:=\frac{1}{n}\sum\limits_{k=1}^{n} \widetilde{Y}_k(t_i), \\
\textrm{Var}(Y(t_i)|Y_0)\label{v}&\approx& \textrm{Var}(\widetilde{Y}(t_i)|Y_0) \approx \hat{v}_{t_i}:=\frac{1}{n-1}\sum\limits_{k=1}^{n} \left( \widetilde{Y}_k(t_i)-\hat{m}_{t_i} \right)^2,
\end{eqnarray}
where $\widetilde{Y}_k(t_i)$ denotes the $k$-th simulated value of $Y(t_i)$ under each considered numerical method, respectively.
We denote by $\text{RE}(\hat{m}_{t_i})$ and $\text{RE}(\hat{v}_{t_i})$ the relative biases \eqref{cond_Bias_mean} and \eqref{cond_Bias_variance}, estimated replacing $\mathbb{E}[\widetilde{Y}(t_i)|Y_0]$ and $\textrm{Var}(\widetilde{Y}(t_i)|Y_0)$ with the sample mean \eqref{m} and variance \eqref{v}, respectively. To investigate the asymptotic case, we fix $t_i=100$ and denote by $\text{RE}(\hat{m}_{100})$ and $\text{RE}(\hat{v}_{100})$ the relative biases \eqref{asym_bias_mean} and \eqref{asym_bias_variance}, estimated replacing $\mathbb{E}[\widetilde{Y}_\infty]$ and $\textrm{Var}(\widetilde{Y}_\infty)$ with $\hat{m}_{100}$ and $\hat{v}_{100}$, respectively.

In the top and bottom left panels of Figure \ref{Fig_moments_cond}, we fix $t_i=15$ and report the true conditional mean $\mathbb{E}[Y(15)|Y_0]$ \eqref{cond_mean} and variance $\textrm{Var}(Y(15)|Y_0)$ \eqref{cond_var} (grey horizontal lines), the theoretical conditional means $\mathbb{E}[\widetilde{Y}(15)|Y_0]$ \eqref{EZi} and variances $\textrm{Var}(\widetilde{Y}(15)|Y_0)$ \eqref{VZi} of the numerical methods as a function of the time step $\Delta$ and their estimated values (symbols) $\hat{m}_{15}$ \eqref{m} and $\hat{v}_{15}$ \eqref{v}, derived for $\Delta=0.25,0.5,0.75,1$. We calculate the sample moments from $n=10^7$ simulations of $Y(15)$, for $\mu=1$, $\tau=5$, $\sigma=0.2$ and $Y_0=10$. In the middle and bottom right panels of Figure \ref{Fig_moments_cond}, we report the true asymptotic mean $\mathbb{E}[Y_\infty]$ \eqref{statE} and variance $\textrm{Var}(Y_\infty)$ \eqref{statV} (grey horizontal lines), the theoretical asymptotic means $\mathbb{E}[\widetilde{Y}_\infty]$ \eqref{lim_EM_1}, \eqref{lim_E_1}-\eqref{lim_E_4}, \eqref{lim_E_Lin}, \eqref{lim_E_Log} and variances $\textrm{Var}(\widetilde{Y}_\infty)$ \eqref{limVarEM}, \eqref{limVarM}, \eqref{limVar}, \eqref{limVar4}, \eqref{limVarLin}, \eqref{limVarLog} as a function of the time step $\Delta$ and their estimated values (symbols) $\hat{m}_{100}$ \eqref{m} and $\hat{v}_{100}$ \eqref{v}, derived for $\Delta=0.25,0.5,0.75,1$. The corresponding relative biases $\text{RE}(\hat{m}_{15})$, $\text{RE}(\hat{v}_{15})$, $\text{RE}(\hat{m}_{100})$ and $\text{RE}(\hat{v}_{100})$, 
for $\Delta=0.5$ and $\Delta=1$ are reported in percentage in Table~\ref{table2}. The quantities obtained through numerical simulations are in agreement with the theoretical ones. Moreover, we verified that there is no noteworthy difference in the standard deviations of the estimated values across the different numerical schemes.

\subsection{Stationary density}

As a further illustration, we investigate the stationary distribution of the IGBM. 
Under the conditions $\sigma^2\tau<2$ and $\mu>0$, the stationary distribution of $Y$ exists and is an inverse gamma distribution \cite{Barone2005,Donofrio2018,Sorensen2007,Zhao2009} with mean \eqref{statE} and variance \eqref{statV}. The probability density function of the stationary distribution of $Y$, which we denote by $f_{Y_\infty}$, is given by 
\vspace{-0.1cm}\begin{equation}\label{InvGamma}
f_{Y_\infty}(y;\alpha,\beta) :=\frac{\beta^\alpha}{\Gamma(\alpha)}y^{-\alpha-1}e^{-\beta/y},
\end{equation}
where $\Gamma(\cdot)$ denotes the gamma function, $\alpha=1+2/\sigma^2\tau$ and $\beta=2\mu/\sigma^2$. 

In Figure \ref{SP1_2_3}, we report the true stationary density $f_{Y_\infty}$ \eqref{InvGamma} (grey solid lines) and the densities $\hat{f}_{{Y}_\infty}$, estimated from $n=10^7$ simulated values of $Y(100)$, for $\mu=1$, $\tau=5$, $\sigma=0.55$ and $Y_0=10$, using the different schemes. The densities are calculated with a kernel density estimator, i.e.,
\begin{equation*}
f_{Y_\infty}(y) \approx \hat{f}_{{Y}_\infty}(y):=\frac{1}{n h}\sum\limits_{k=1}^{n} \mathcal{K}\left( \frac{y-\widetilde{Y}_k(100)}{h} \right), 
\end{equation*}
where the bandwidth $h$ is a smoothing parameter and $\mathcal{K}$ is a kernel function (here Gaussian). If $\Delta=0.5$ (left panels), the Strang and ODE schemes (bottom left panel) accurately preserve the stationary density, while the other schemes (top left panel) yield estimates that deviate from the true density. This discrepancy increases as $\Delta$ increases (top right panel), while the Strang and ODE schemes (bottom right panel) still yield satisfactory estimates. 

To quantify the distance between the true and the estimated densities under the considered numerical schemes for different time steps, we consider their Kullback-Leibler (KL) divergences given by
\begin{equation}\label{KL}
\text{KL}:=
\int f_{Y_\infty}(y) \ \text{log} \left( \frac{f_{Y_\infty}(y)}{\hat{f}_{{Y}_\infty}(y)}\right) dy, 
\end{equation}
where the integral is approximated using trapezoidal integration. 
The results shown in Figure \ref{SP1_2_3} are confirmed by the KL divergences \eqref{KL} reported in Table \ref{table2}. In particular, the best performance is achieved by the log-ODE method, which yields a very accurate estimate of the stationary density, even for $\Delta=1$, and even though for $\tau=5$ and $\sigma=0.55$ the first Strang scheme introduces a smaller bias in the asymptotic variance (see Subsection \ref{subsec:4:1:1:new}). Moreover, for the chosen parameter setting, the Strang schemes yield slightly better estimates of the stationary density than the piecewise linear method, and the Lie-Trotter schemes outperform the Euler-Maruyama and Milstein methods. 

\vspace{1cm}
\begin{table}[H]
	{\footnotesize  
		\caption{Comparison of theoretical quantities with simulated values ($n=10^7$). We report relative conditional mean and variance biases \eqref{cond_Bias_mean} and \eqref{cond_Bias_variance}, asymptotic mean and variance biases \eqref{asym_bias_mean} and \eqref{asym_bias_variance} (in parentheses) and $1000$ times the KL divergences \eqref{KL} for $\Delta=0.5$ and $\Delta=1$. The parameters are $Y_0=10$, $\mu=1$, $\tau=5$, $\sigma=0.2$ (REs) and  $\sigma=0.55$ (KLs).}
		\label{table2}	
		\begin{center}\vspace{-0.3cm}
			\scalebox{1.1}{
				\begin{tabular}{|c!{\vrule width 1pt}c|c|c|c|c|} 
					\hline 
					$\mathbf{\Delta=0.5}$ & RE($\hat{m}_{15}$) in $\%$age & RE($\hat{v}_{15}$) in $\%$age & RE($\hat{m}_{100}$) in $\%$age & RE($\hat{v}_{100}$) in $\%$age & $1000 \cdot\text{KL}$ \\			
					\noalign{\hrule height 1pt} 
					$\textrm{E}$ & $-0.694 \ (-0.705)$ & $4.425 \ (4.4)$ & $-0.001 \ (0)$ & $5.897 \ (5.882)$ & $0.925$ \\
					$\textrm{M}$ & $-0.694 \ (-0.705)$ & $5.602 \ (5.586)$ & $-0.001 \ (0)$ & $7.092 \ (7.067)$ & $1.124$ \\
					$\textrm{L1}$  & $-4.44 \ (-4.45)$ & $0.81 \ (0.786)$ & $-4.917 \ (-4.917)$ & $-0.191 (-0.216)$ & $0.194$ \\
					$\textrm{L2}$ & $4.611 \ (4.601)$ & $-1.163 \ (-1.187)$ & $5.083 \ (5.083)$ & $-0.191 \ (-0.216)$ & $0.445$ \\
					$\textrm{S1}$  & $0.085 \ (0.075)$ & $-0.18 \ (-0.205)$ & $0.083 \ (0.083)$ & $-0.191 \ (-0.216)$ & $0.005$ \\
					$\textrm{S2}$ & $-0.039 \ (-0.038)$ & $0.228 \ (0.204)$ & $-0.025 \ (-0.042)$ & $0.281 \ (0.233)$ & $0.009$ \\
					$\textrm{Lin}$ & $-0.152 \ (-0.151)$ & $-0.33 \ (-0.335)$ & $-0.167 \ (-0.166)$ & $-0.391 \ (-0.415)$ & $0.045$ \\
					$\textrm{Log}$ & $-0.008 \ (-0.0001)$ & $-0.094 \ (-0.001)$ & $0.01 \ (-0.0001)$ & $0.07 \ (-0.001)$ & $0.003$ \\
					\hline\hline
					$\mathbf{\Delta=1}$ & RE($\hat{m}_{15}$) in $\%$age & RE($\hat{v}_{15}$) in $\%$age & RE($\hat{m}_{100}$) in $\%$age & RE($\hat{v}_{100}$) in $\%$age & $1000 \cdot\text{KL}$ \\
					\noalign{\hrule height 1pt} 
					$\textrm{E}$ & $-1.384 \ (-1.391)$ & $9.315 \ (9.296)$ & $0.006 \ (0)$ & $12.461 \ (12.5)$ & $5.139$ \\
					$\textrm{M}$ & $-1.383 \ (-1.391)$ & $11.796 \ (11.807)$ & $0.003 \ (0)$ & $14.957 \ (15.038)$ & $3.743$ \\
					$\textrm{L1}$  & $-8.742 \ (-8.75)$ & $1.176 \ (1.169)$ & $-9.663 \ (-9.667)$ & $-0.921 \ (-0.862)$ & $0.639$ \\
					$\textrm{L2}$ & $9.361 \ (9.353)$ & $-2.762 \ (-2.769)$ & $10.337 \ (10.333)$ & $-0.921 \ (-0.862)$ & $2.981$ \\
					$\textrm{S1}$  & $0.31 \ (0.302)$ & $-0.808 \ (-0.815)$ & $0.337 \ (0.333)$ & $-0.921 \ (-0.862)$ & $0.069$ \\
					$\textrm{S2}$ & $-0.129 \ (-0.151)$ & $0.91 \ (0.814)$ & $-0.164 \ (-0.166)$ & $0.879 \ (0.928)$ & $0.071$ \\
					$\textrm{Lin}$ & $-0.286 \ (-0.301)$ & $-0.745 \ (-0.802)$ & $-0.329 \ (-0.332)$ & $-1.05 \ (-0.992)$ & $0.208$ \\
					$\textrm{Log}$ & $0.003 \ (-0.0002)$ & $-0.005 \ (-0.003)$ & $0.009 \ (-0.0002)$ & $0.019 \ (-0.004)$ & $0.001$ \\
					\hline
			\end{tabular}}
		\end{center}
	}
\end{table}

\begin{figure}[H]
	\centering
	\includegraphics[width=0.9\textwidth]{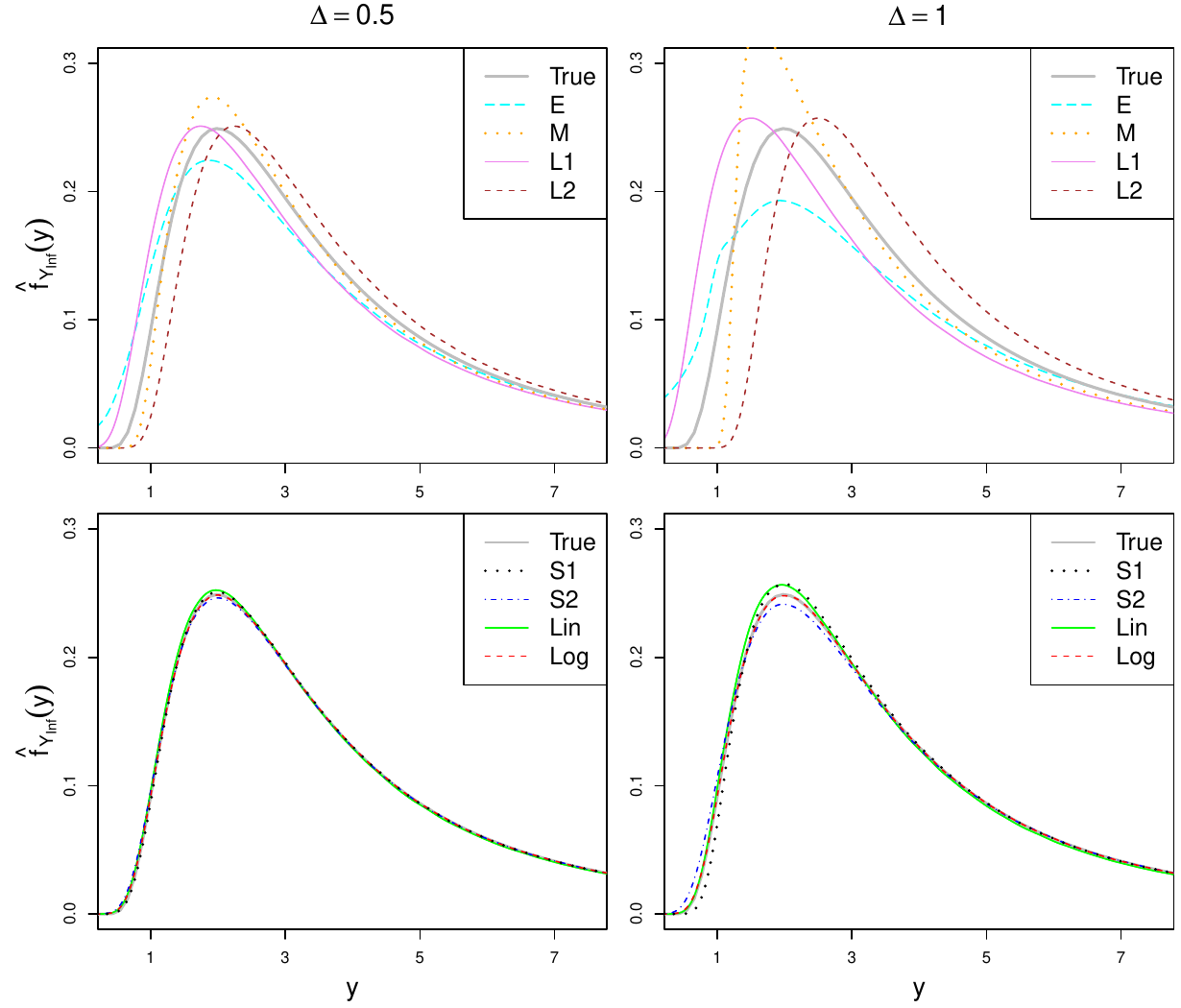}\\
	\caption{Comparison of the stationary density $f_{Y_\infty}$ (grey solid lines) and the estimated densities $\hat{f}_{{Y}_\infty}$ based on $n=10^7$ simulations of $Y(100)$, generated with the different numerical schemes, for $\Delta=0.5$ (left panels) and $\Delta=1$ (right panels). The underlying parameters are $\mu=1$, $\tau=5$, $\sigma=0.55$ and $Y_0=10$. The corresponding KL divergences \eqref{KL} are reported in Table \ref{table2}.}
	\label{SP1_2_3}
\end{figure}

\vspace{-0.6cm}
\subsection{Boundary properties}
\vspace{-0.1cm}

An illustration of the preservation of the boundary properties by the splitting and ODE schemes is provided in Figure \ref{4Paths}, where we report trajectories generated with the first Lie-Trotter, first Strang, piecewise linear and log-ODE schemes when the boundary $0$ is of entrance (top panel), unattainable and attracting (middle panel) and exit (bottom panel) type. In particular, we use $\mu=-0.5$, $0$ and $0.5$, respectively, $\tau=5$ and $\sigma=1$.

\begin{figure}
	\centering
	\includegraphics[width=1.0\textwidth]{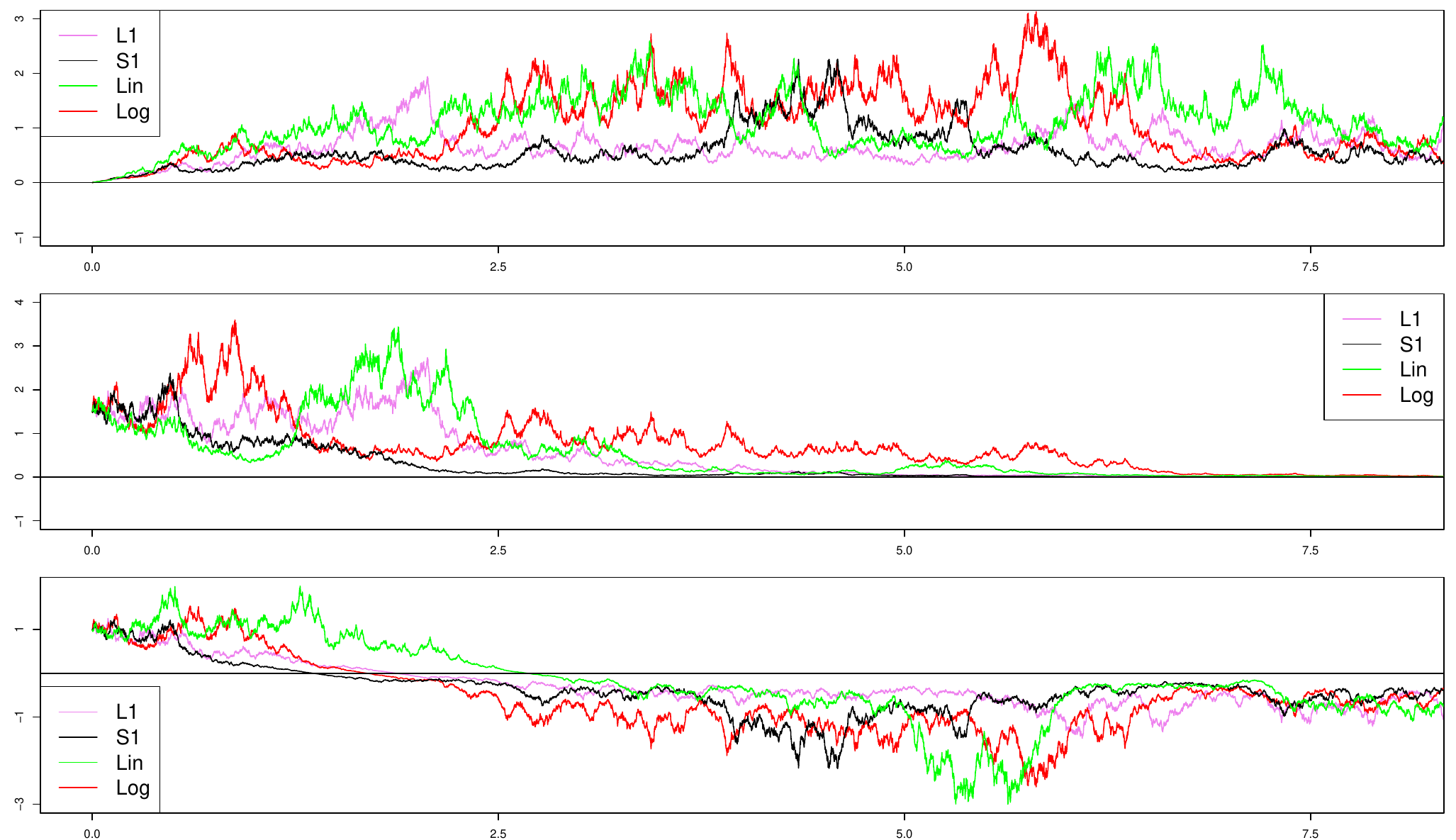}\\
	\caption{Trajectories of the IGBM generated with the first Lie-Trotter, first Strang, piecewise linear and log-ODE schemes, for  $\tau=5$ and $\sigma=1$. The parameter $\mu$ is chosen  such that the boundary $0$ is of entrance (top panel, $\mu=0.5$), unattainable and attracting (middle panel, $\mu=0$) and exit (bottom panel, $\mu=-0.5$) type.} 
	\label{4Paths}
\end{figure}

\vspace{-0.2cm}
\subsection{Crossing probability}
\vspace{-0.1cm}

As a further illustration of the boundary behaviour, we investigate the probability that the process $Y$ crosses the boundary $0$ in a fixed time interval $(0,t_{\textrm{max}}]$, with $t_{\textrm{max}}>0$ and $Y_0>0$. We define
\begin{equation}\label{T}
T:=\text{inf}\{ t>0 : Y(t) \leq 0 \}
\end{equation}
as the first passage (hitting) time of $Y$ through $0$, and 
estimate the probability that $T<t_{\textrm{max}}$ as follows  
\vspace{-0.2cm}\begin{equation}\label{P}
\mathbb{P}(T<t_\textrm{max}) \approx \hat{F}_{T}(t_\textrm{max}) := \frac{1}{n} \sum_{k=1}^{n}  \mathbbm{1}_{ \{ T_{k}<t_{\textrm{max}} \}}, 
\end{equation}
where $T_{k}$ denotes the crossing time \eqref{T}, which is obtained from the $k$-th simulated path of $Y$ and $\mathbbm{1}_A$ denotes the indicator function of the set $A$.
We are interested in situations where the process is in a high noisy regime, i.e., it is perturbed by a large noise intensity $\sigma$ and is not in its stationary~regime. 

In Figure \ref{Pmu}, we report $\mathbb{P}(T<t_\textrm{max})$, estimated from $n=10^6$ simulated trajectories under the different numerical schemes, as a function of ${\mu}$, for $\sigma=5$, $\tau=5$, $Y_0=1$, $t_{\text{max}}=0.5$ and different choices of the time step, namely $\Delta=0.01$ (left panel), $\Delta=0.025$ (middle panel) and $\Delta=0.05$ (right panel). The threshold $0$ is of exit, unattainable and attracting (denoted by dashed grey vertical lines) or entrance type depending on whether ${\mu}<0$, ${\mu}=0$ or ${\mu}>0$, respectively. Note that the functions obtained under the splitting and ODE schemes lie close to each other, in spite of the large value of $\sigma$.
When the boundary $0$ is of entrance or unattainable and attracting type, it is known that $\mathbb{P}(T<t_\textrm{max})=0$ for all values of $t_{\text{max}}$. However, only the splitting and ODE schemes correctly preserve this property, while the Euler-Maruyama method drastically fails for all considered values of $\Delta$ and the Milstein scheme only preserves it  for small values of $\Delta$ (left and middle panels). The latter is in agreement with condition \eqref{cond_Delta}. Consider, e.g., $y=Y_0=1$. Then $\Delta<5/122 \approx 0.0402$ and $\Delta<5/127 \approx 0.0394$ is required in the entrance or unattainable and attracting case, respectively. 
In the exit scenario, the probabilities obtained from the Euler-Maruyama and Milstein schemes lie above those obtained from the splitting and ODE schemes. This suggests that the Euler-Maruyama and Milstein methods yield trajectories that  exit from $[0,+\infty)$ faster than those generated from the other schemes. 
Similar results are obtained when studying these probabilities as a function of $t_{\text{max}}$ for fixed ${\mu}$. Moreover, independent of the type of boundary behaviour, the crossing probabilities obtained from the Strang splitting and log-ODE schemes seem not to vary significantly as $\Delta$ increases (a few undetected crossings may occur). This suggests their reliability even for large time steps, while those obtained from the Euler-Maruyama and Milstein schemes change for different choices of $\Delta$. The crossing probabilities derived under the Lie-Trotter and piecewise linear methods deviate slightly as $\Delta$ is increased, the latter one performing a bit better.

\begin{figure}
	\centering
	\centering
	\includegraphics[width=1.0\textwidth]{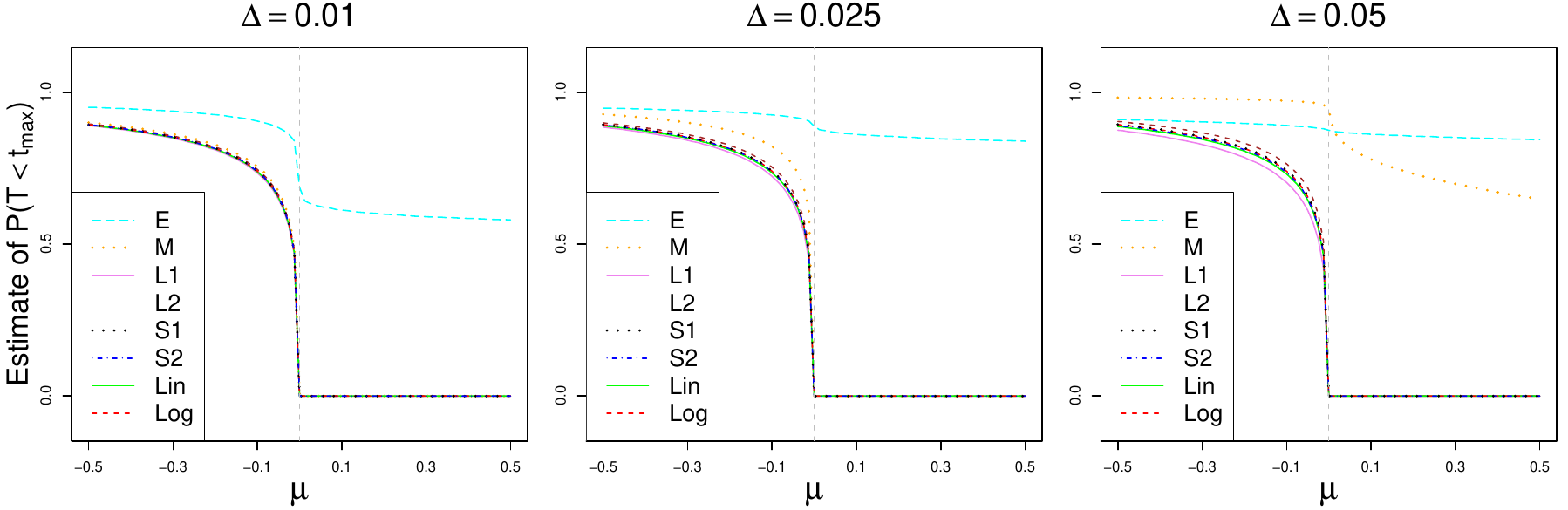}\\
	\caption{Probability $\mathbb{P}(T<t_\textrm{max})$ \eqref{P}, estimated from n=$10^6$ simulated trajectories under the different numerical schemes, as a function of ${\mu}$, for different choices of the time step, namely $\Delta=0.01$ (left panel), $\Delta=0.025$ (middle panel) and $\Delta=0.05$ (right panel), $t_{\text{max}}=0.5$, $\tau=\sigma=5$ and $Y_0=1$. The boundary $0$ is of entrance, unattainable and attracting or exit type depending on whether $\mu>0$, $\mu=0$ (denoted by dashed grey vertical lines) or $\mu<0$, respectively.}
	\label{Pmu}
\end{figure}

\section{Conclusion}
\label{sec6}

Any numerical method, constructed to approximate a process of interest, should preserve its qualitative properties. Here, we focus on the IGBM, a process characterised by a constant inhomogeneous term, commonly applied in mathematical finance, neuroscience and other fields. We compare two Lie-Trotter splitting schemes, two Strang splitting schemes and two schemes based on the ODE approach (the classical piecewise linear method \cite{Wong1965_2} and the recently introduced log-ODE method \cite{Foster2020}) with the frequently applied Euler-Maruyama and Milstein methods both analytically and via simulations. 

We prove that, in contrast to the frequently applied methods, the splitting and ODE schemes preserve the different boundary properties of the IGBM, independently of the choice of the time discretisation step. We also investigate through simulations the probability that the process crosses the lower boundary. Compared to the splitting and ODE schemes, the Euler-Maruyama and Milstein methods suggest not only a positive crossing probability in the entrance or unattainable and attracting case, but also higher crossing probabilities in the exit scenario. 

Moreover, we provide closed-form expressions for the conditional and asymptotic means and variances of the considered numerical solutions, and analyse the resulting biases with respect to the true quantities. The Euler-Maruyama and Milstein schemes are the only methods having an asymptotically unbiased mean (if an extra condition, unrelated to the features of the model, is fulfilled). However, the splitting and ODE schemes yield better approximations of the variance of the process, and do not require extra conditions for the existence of the asymptotic quantities. We observe that the Strang splitting schemes clearly outperform the Lie-Trotter splitting schemes in terms of preserving the mean and stationary density of the process, and that the log-ODE method performs better than the piecewise linear method throughout. Both the Strang and log-ODE schemes show a solid performance. The biases introduced by the log-ODE method are even smaller than that of the Strang schemes for many relevant parameter configurations. However, the drawback of the log-ODE method is that its mean bias depends on the noise parameter $\sigma$, and, consequently, can deteriorate for large values of $\sigma$. In this case, the two Strang methods, which perform comparably good throughout, may be better alternatives.

Moreover, we emphasise that the first Strang scheme requires almost the same computational effort as the Lie-Trotter and standard methods, while the second Strang and log-ODE schemes are more computationally expensive. In particular, they require to generate two random numbers in each iteration and rely on more function evaluations.

Guaranteeing a correct behaviour of the simulated process near or at the boundary is important in a variety of applications such as optimal stopping problems or positive asset pricing models.  Moreover, having explicit closed-form expressions for the first two conditional and asymptotic moments of the numerical solutions may, for example, play an important role in moment based statistical inference \cite{Ditlevsen2019,Sorensen2007}. All schemes yield biased moments which will effect the inferential approaches. Knowing them explicitly may help to adjust the inferential procedure accordingly. The explicit closed-form expressions also allow for a direct control of the respective simulation accuracy through the time discretisation step. There is a trade-off between computation time and quality of the simulation. To achieve a reasonable computation time, it may be necessary to avoid very small time steps. This becomes particularly important when the numerical method is embedded, for example, in a simulation-based inference method \cite{Buckwar2019,Voss2014}.

The considered equation, its properties and their analysis are also meant as a contribution to extend the range of qualitative features that characterise the quality of numerical methods. The presented results on the IGBM may be extended to other numerical methods and to a wider class of SDEs with similar features. For example, one may derive the exact moments of numerical solutions of other Pearson diffusions \cite{Sorensen2007} and analyse their boundary behaviour in a similar fashion. The presented analysis may also be extended to multi-dimensional versions of the IGBM, and to a broader class of equations, e.g., via adapted linearisation and diagonalisation procedures. Finally, the construction of a boundary preserving numerical method for the IGBM which has at least an asymptotically unbiased mean, still remains an open problem.



\newpage
\appendix
\section{Proof of \Cref{Prop_cond_1}}
\label{appA}

\begin{proof}	
	Since the underlying $X_j$, $j=1,\ldots,k$, are iid, the mean and variance of $W_k$, $k \in \mathbb{N}$, are given by
	\begin{eqnarray*}\label{app01}
		\label{EWk}\mathbb{E}[W_k]&=&\mu_x^k, \\
		\label{VarWk}\textrm{Var}(W_k)&=&\textrm{Var}\left( \prod\limits_{j=1}^k X_j \right)=\prod\limits_{j=1}^k \mathbb{E}[X_j^2] - \prod\limits_{j=1}^k \mathbb{E}[X_j]^2=r^k-\mu_x^{2k}.
	\end{eqnarray*}
	Using the independence of $W_k$ and $H_{k+1}$, the mean of $Z_i$ conditioned on $Z_0$ is given by \eqref{EZi}.
	
	To compute the variance of $Z_i$ conditioned on $Z_0$, using the independence of $W_k$ and $H_{k+1}$, we have that
	\begin{eqnarray}\label{V}
	\textrm{Var}(W_kH_{k+1})&=&r^kr_h-\mu_x^{2k}\mu_h^2.
	\end{eqnarray}
	Further, using the independence assumption of $W_k$ and $H_{k+1}$, we obtain for $k<l$
	\begin{eqnarray}
	\hspace{-5cm}\label{Cov1}\textrm{Cov}(W_l,W_kH_{k+1})&=& \nonumber \mathbb{E}[W_l W_k H_{k+1}]-\mathbb{E}[W_l]\mathbb{E}[W_kH_{k+1}] \\ &\stackrel{k<l}{=}& r^k \mu_x^{l-k}p - \mu_x^{l+k}\mu_h ,\\ 
	\label{Cov2}\textrm{Cov}(W_lH_{l+1},W_kH_{k+1})&=& \nonumber \mathbb{E}[ W_l H_{l+1} W_k H_{k+1} ] - \mathbb{E}[W_l H_{l+1}] \mathbb{E}[W_k H_{k+1}] \\&\stackrel{k<l}{=}&\nonumber\mathbb{E}[W_l W_k H_{k+1}]\mathbb{E}[H_{l+1}]-\mathbb{E}[W_l]\mathbb{E}[H_{l+1}]\mathbb{E}[W_k]\mathbb{E}[H_{k+1}]\\
	&=&\mu_h r^k \mu_x^{l-k}p-\mu_x^{l+k}\mu_h^2.
	\end{eqnarray}
	Hence, the conditional variance of $Z_i$ given $Z_0$ is given by
	\begin{eqnarray}\label{condVarZ}
	\textrm{Var}(Z_i|Z_0)\nonumber&=&Z_0^2\textrm{Var}(W_i)+c_1^2\textrm{Var}\left(\sum_{k=0}^I W_kH_{k+1}\right)+2c_1Z_0\textrm{Cov}(W_i,\sum_{k=0}^I W_k H_{k+1})\\
	\nonumber&=&  Z_0^2\textrm{Var}(W_i)+c_1^2\left[\sum_{k=0}^I \textrm{Var}(W_k H_{k+1})+2\sum_{l=1}^I\sum_{k=0}^{l-1}\textrm{Cov}(W_l H_{l+1},W_k W_{k+1})\right]\\&&\hspace{0.5cm}+2c_1Z_0\sum_{k=0}^I\textrm{Cov}(W_i,W_k H_{k+1}),
	\end{eqnarray}
	yielding \eqref{VZi} after plugging \eqref{V}, \eqref{Cov1} and \eqref{Cov2} into \eqref{condVarZ}.
\end{proof}

\vspace{-0.5cm}
\section{Proof of \Cref{Lemma_cond_1}}
\label{appB}

\vspace{-0.3cm}
\begin{proof}
	Define $W_k^l:=\prod\limits_{j=k}^{l}X_j$, with $\mathbb{E}[W_k^l]=\mu_x^{l-k+1}$. In the following, we use that if $H_{k+1}=1$, then $p=1$, since
	\begin{equation*}
	\mathbb{E}[W_lW_kH_{k+1}] = \mathbb{E}[W_l W_k ]\stackrel{k<l}{=}\mathbb{E}[W_k W_{k+1}^l W_k]= \mathbb{E}[W_k^2]\mathbb{E}[W_{k+1}^l]=r^k\mu_x^{l-k}.
	\end{equation*}\noindent
	Since the Gaussian increments $\xi_j\sim \mathcal{N}(0,\Delta)$ are iid, the Euler-Maruyama scheme $\widetilde{Y}^{\textrm{E}}(t_i)$ \eqref{Ye} can be rewritten as \eqref{Zi} with $W_k$ defined via
	\begin{eqnarray*}
		&&X_j:=\left(1-\frac{\Delta}{\tau}+\sigma\xi_j\right)\sim \mathcal{N}\left(1-\frac{\Delta}{\tau},\sigma^2\Delta\right), \quad j=1,\ldots,k,
	\end{eqnarray*}
	$H_{k+1}=1$, $p=1$
	and the values reported in Table \ref{table1}. Using the property that $\mathbb{E}[\xi_j]=\mathbb{E}[\xi_j^3]=0$, $\mathbb{E}[\xi_j^2]=\Delta$, $\mathbb{E}[\xi_j^4]=3\Delta^2$, the Milstein scheme $\widetilde{Y}^{\textrm{M}}(t_i)$ \eqref{Ym} can be rewritten as \eqref{Zi} with $W_k$ defined~via
	\begin{eqnarray*}
		&&X_j:=\left(1-\frac{\Delta}{\tau}+\sigma\xi_j+(\xi^2_j-\Delta)\frac{\sigma^2}{2}\right), \quad j=1,\ldots,k,
	\end{eqnarray*}
	$H_{k+1}=1$, $p=1$ and the values reported in Table \ref{table1}. 
	
	\newpage
	
	The splitting scheme $\widetilde{Y}^{\textrm{L1}}(t_i)$ \eqref{Y1} can be rewritten as \eqref{Zi} with $W_k$ defined via
	\begin{equation}\label{Wi}
	X_j:=e^{-\left(\frac{1}{\tau}+\frac{\sigma^2}{2}\right)\Delta+\sigma \xi_j}, \quad j=1,\ldots,k,
	\end{equation}
	$H_{k+1}=1$, $p=1$
	and the values reported in Table \ref{table1}.
	Since $\xi_j\sim \mathcal{N}(0,\Delta)$, the random variable 
	$-\left(1/\tau+\sigma^2/2\right)\Delta+\sigma\xi_j \sim \mathcal{N}\left(-\left(1/\tau+\sigma^2/2\right)\Delta,\sigma^2\Delta\right)$,
	and thus the $X_j$ are iid random variables with log-normal distribution, mean $\mu_x$ and second moment $r$ given by
	\begin{eqnarray*}
		\mu_x=\label{meanLi}\mathbb{E}[X_j]&=&e^{-\left(\frac{1}{\tau}+\frac{\sigma^2}{2}\right)\Delta+\frac{1}{2}\sigma^2 \Delta}=e^{-\Delta/\tau},\\
		r=\label{varLi}\mathbb{E}[X_j^2]&=&(e^{\sigma^2\Delta}-1)e^{-2\left(\frac{1}{\tau}+\frac{\sigma^2}{2}\right)\Delta+\sigma^2 \Delta}+e^{-2\Delta/\tau}=e^{\sigma^2\Delta-2\Delta/\tau}.
	\end{eqnarray*}
	Similarly, the splitting schemes $\widetilde{Y}^{\textrm{L2}}(t_i)$ \eqref{Y2} and $\widetilde{Y}^{\textrm{S1}}(t_i)$ \eqref{Y3} can be rewritten as \eqref{Zi} using $X_j$ given by \eqref{Wi}, as for $\widetilde{Y}^{\textrm{L1}}$, and the values reported in Table \ref{table1}. 
	
	Further, since $\varphi_j$ and $\psi_j$ are iid random variables distributed as $\mathcal{N}(0,{\Delta}/{2})$, we have that $\xi_j:=\varphi_j+\psi_j\sim \mathcal{N}(0,\Delta)$. Setting
	\begin{eqnarray*}
		&&X_{j}:=e^{-(\frac{1}{\tau}+\frac{\sigma^2}{2})\Delta+\sigma \xi_{j}}, \quad H_{k+1}:=e^{-(\frac{1}{\tau}+\frac{\sigma^2}{2})\frac{\Delta}{2}+\sigma \psi_{k+1}}, \quad j=1,\ldots,k,
	\end{eqnarray*}\noindent
	using an index shift and splitting off the 0-th element,
	the splitting scheme $\widetilde{Y}^{\textrm{S2}}(t_i)$ \eqref{Y4} can be rewritten as  \eqref{Zi} with the values reported in Table \ref{table1}. In particular, note that $W_{k+1}^{k+1}=X_{k+1}=H_{k+1}M_{k+1}$ with $M_{k+1}$ having mean $\mu_x^{1/2}$ and being independent from $H_{k+1}$. Thus,
	\begin{eqnarray*}
	\mathbb{E}[W_lW_kH_{k+1}]&\stackrel{k<l}{=}&\mathbb{E}[W_k^2 W_{k+1}^{k+1} W_{k+2}^{l} H_{k+1}]=\mathbb{E}[W_k^2]\mathbb{E}[W_{k+2}^l] \mathbb{E}[H_{k+1}^2]\mathbb{E}[M_{k+1}] \\ &=& r^k \mu_x^{l-k-1}r_h \mu_x^{1/2}=r^k \mu_x^{l-k} p,
	\end{eqnarray*}
	with  $p=r_h \mu_x^{-1/2}$.
	
	The piecewise linear method $\widetilde{Y}^{\textrm{Lin}}(t_i)$ \eqref{Lin} can be rewritten as as \eqref{Zi} with $W_k$ defined via $X_j$, $j=1,\ldots,k$, as in \eqref{Wi}, 
	\begin{equation*}
	H_{k+1}:= \left( \frac{e^{-(\frac{1}{\tau}+\frac{\sigma^2}{2})\Delta+\sigma \xi_{k+1}}-1}{-(\frac{1}{\tau}+\frac{\sigma^2}{2})\Delta+\sigma \xi_{k+1}} \right),
	\end{equation*}
	and the values reported in Table \ref{table1}. In particular, the mean of $H_{k+1}$ is given by \eqref{L_delta_tau_sigma}, since
	\begin{eqnarray*}
		\nonumber \mu_h&=&\mathbb{E}[H_{k+1}]=\int\limits_{0}^{1} \mathbb{E}\left[ e^{s\left( -\bigl( \frac{1}{\tau}+\frac{\sigma^2}{2} \bigr)\Delta + \sigma \xi_{k+1} \right)} \right] ds=\int\limits_{0}^{1} e^{ -\left( \frac{1}{\tau}+\frac{\sigma^2}{2} \right)\Delta s + \frac{\sigma^2}{2} s^2 \Delta } \ ds = L_{\Delta,\tau,\sigma}. 
	\end{eqnarray*}
	Moreover, the second moment reads as
	\begin{eqnarray*}\label{barL_delta_tau_sigma}
		\nonumber r_h&=&\mathbb{E}[H_{k+1}^2] =  \int\limits_{0}^{1} \int\limits_{0}^{1} \mathbb{E}\left[ e^{r\left( -\left( \frac{1}{\tau}+\frac{\sigma^2}{2} \right)\Delta + \sigma \xi_{k+1} \right)} e^{s\left(-\left( \frac{1}{\tau}+\frac{\sigma^2}{2} \right)\Delta + \sigma \xi_{k+1} \right)} \right] \ dr ds  \\
		&=& \int\limits_{0}^{1} \int\limits_{0}^{1}  e^{-\left( \frac{1}{\tau}+\frac{\sigma^2}{2} \right)\Delta (r+s)+ \frac{\sigma^2}{2}\Delta \left( r^2+s^2+2rs \right) } \ dr ds = \bar{L}_{\Delta,\tau,\sigma},
	\end{eqnarray*}
	where
	{\scriptsize{\begin{eqnarray}
			\bar{L}_{\Delta,\tau,\sigma}&:=&\frac{1}{2\sigma^3\Delta}\exp\left( \frac{- \left( \frac{1}{\tau}+\frac{\sigma^2}{2} \right)^2 \Delta }{2\sigma^2} \right) \Bigg\{  -2\sigma e^{\left( \frac{1}{\tau}+\frac{\sigma^2}{2} \right)\frac{\Delta}{2} \left[ -4+\left( \frac{1}{\tau\sigma^2}+\frac{1}{2} \right) \right] }  \left( e^{\left( \frac{1}{\tau}+\frac{\sigma^2}{2} \right)2\Delta} + e^{2\sigma^2\Delta} -2 e^{\left( \frac{1}{\tau}+\frac{\sigma^2}{2} \right)\Delta+ \frac{\sigma^2}{2}\Delta } \right)    \\ 
			&& \nonumber \hspace{-1.7cm} + \sqrt{2\pi\Delta} \left( \text{erfi}\left[ \frac{\left( \frac{1}{\tau} + \frac{\sigma^2}{2}  \right)\sqrt{\Delta}}{\sigma\sqrt{2}} \right] \left( \frac{1}{\tau}+\frac{\sigma^2}{2} \right)  + \text{erfi}\left[ \frac{\left( \frac{1}{\tau} - \frac{3\sigma^2}{2}  \right)\sqrt{\Delta}}{\sigma\sqrt{2}} \right] \left( \left( \frac{1}{\tau}+\frac{\sigma^2}{2} \right) -2\sigma^2 \right)  + \text{erfi}\left[ \frac{\left( \frac{1}{\tau} - \frac{\sigma^2}{2}  \right)\sqrt{\Delta}}{\sigma\sqrt{2}} \right] 2 \left( -\frac{1}{\tau}+\frac{\sigma^2}{2} \right) \right) \Bigg\}.
			\end{eqnarray}}}\noindent
	In addition, $p=\widetilde{L}_{\Delta,\tau,\sigma}$, where
	\begin{equation}\label{tildeL_delta_tau_sigma} \hspace{-0.5cm}
	\widetilde{L}_{\Delta,\tau,\sigma}:=\frac{\sqrt{\pi}}{\sigma \sqrt{2\Delta}}\exp\left(\frac{-\left( \frac{1}{\tau}-\frac{\sigma^2}{2} \right)^2\Delta}{2\sigma^2} \right)\left( \text{erfi}\left[ \frac{\left( \frac{1}{\tau}-\frac{\sigma^2}{2} \right)\sqrt{\Delta}}{\sigma \sqrt{2}} \right] + \text{erfi}\left[ \frac{\left(- \frac{1}{\tau}+\frac{3\sigma^2}{2} \right)\sqrt{\Delta}}{\sigma \sqrt{2}} \right] \right),
	\end{equation}
	since
	\begin{eqnarray*}
		\mathbb{E}[W_l W_k H_{k+1}]&=&\mathbb{E}\left[ e^{-\left( \frac{1}{\tau}+\frac{\sigma^2}{2} \right)\Delta(l+k) + \sigma \left( \sum\limits_{m=1}^{l}\xi_m + \sum\limits_{j=1}^{k} \xi_j \right) } \int\limits_{0}^{1} e^{s \left( -\left( \frac{1}{\tau}+\frac{\sigma^2}{2} \right)\Delta + \sigma \xi_{k+1} \right)}  ds \ \right] \\ &=& \int\limits_{0}^{1} \mathbb{E}\left[ e^{-\left( \frac{1}{\tau}+\frac{\sigma^2}{2} \right)\Delta (l+k+s)+\sigma\left( \sum\limits_{m=1}^{l} \xi_m + \sum\limits_{j=1}^{k} \xi_j + s \xi_{k+1} \right) } \right] \ ds \\ 
		&=& \int\limits_{0}^{1}  e^{-\left( \frac{1}{\tau}+\frac{\sigma^2}{2} \right)\Delta (l+k+s)+ \frac{\sigma^2}{2}\Delta (l+3k+2s+s^2) } \ ds \\ &=&e^{-\frac{1}{\tau}(l+k)\Delta + \sigma^2 t_k } \int\limits_{0}^{1} e^{-\left( \frac{1}{\tau}-\frac{\sigma^2}{2} \right)\Delta s + \frac{\sigma^2}{2}\Delta s^2} \ ds = r^k \mu_x^{l-k} \widetilde{L}_{\Delta,\tau,\sigma}.
	\end{eqnarray*}
	
	The log-ODE method $\widetilde{Y}^{\textrm{Log}}(t_i)$ \eqref{Log} can be treated in the same way, except for considering 
	\begin{equation*}
	H_{k+1}:= \left( \frac{e^{-(\frac{1}{\tau}+\frac{\sigma^2}{2})\Delta+\sigma \xi_{k+1}}-1}{-(\frac{1}{\tau}+\frac{\sigma^2}{2})\Delta+\sigma \xi_{k+1}} \right)\left( 1-\sigma \rho_{k+1}+\sigma^2 \left[ \frac{3}{5} \rho_{k+1}^2 +\frac{\Delta}{30} \right] \right).
	\end{equation*}
	Using the independence of $\xi_{k+1}$ and $\rho_{k+1}$, and recalling that $\rho_{k+1}\sim\mathcal{N}(0,\Delta/12)$, the mean and second moment of $H_{k+1}$ are given by
	\begin{eqnarray}
	\label{K_delta_tau_sigma} \mu_h&=&\mathbb{E}[H_{k+1}]=K_{\Delta,\tau,\sigma}:=L_{\Delta,\tau,\sigma} \left( 1+\sigma^2 \frac{\Delta}{12} \right), \\
	\label{barK_delta_tau_sigma}r_h&=&\mathbb{E}[H^2_{k+1}]=\bar{K}_{\Delta,\tau,\sigma}:=\bar{L}_{\Delta,\tau,\sigma} \left( 1+\sigma^2\frac{\Delta}{4} +\sigma^4\Delta^2 \frac{43}{3600} \right),
	\end{eqnarray}
	where $L_{\Delta,\tau,\sigma}$ and $\bar{L}_{\Delta,\tau,\sigma}$ are as in \eqref{L_delta_tau_sigma} and \eqref{barL_delta_tau_sigma}, respectively. Moreover, we have that
	\begin{equation*}
	\mathbb{E}[W_lW_kH_{k+1}]=r^k\mu_x^{l-k}\widetilde{L}_{\Delta,\tau,\sigma}\left( 1+ \sigma^2 \frac{\Delta}{12} \right),
	\end{equation*}
	where $\widetilde{L}_{\Delta,\tau,\sigma}$ is as in \eqref{tildeL_delta_tau_sigma}, and thus
	\begin{equation}\label{tildeK_delta_tau_sigma}
	p=\widetilde{K}_{\Delta,\tau,\sigma}:= \widetilde{L}_{\Delta,\tau,\sigma}\left( 1+ \sigma^2 \frac{\Delta}{12} \right).
	\end{equation}
\end{proof}

\newpage

\section{Proof of \Cref{Prop_asym_1}}
\label{appC}

\begin{proof}
	When letting $i\to \infty$ (and thus $I \to \infty$), $\sum\limits_{k=1}^I \mu_x^k$ converges to $\mu_x/(1-\mu_x)$ if and only if $|\mu_x|<1$. Under the same condition, $\mu_x^i\to 0$, yielding the asymptotic mean of $Z_i$ given by \eqref{AEZi}.
	
	Letting $i\to\infty$ (and thus $I \to \infty$), $\sum\limits_{k=1}^I (r^{k}r_h-\mu_x^{2k}\mu_h^2)$ converges if and only if $r\in(0,1)$ and $|\mu_x|<1$. Hence, under these conditions, 
	\begin{eqnarray*}
		&&\lim_{i\to\infty}Z_0^2(r^i-\mu_x^{2i})= 0,\\
		&&\lim_{i\to\infty}\sum_{k=0}^{I}r^{k}r_h-\mu_x^{2k}\mu_h^2=-\frac{r_h}{r-1}+\frac{\mu_h^2}{(\mu_x^2-1)},\\
		&&\lim_{i\to\infty}\sum_{l=1}^{I}\sum_{k=0}^{l-1} \mu_h r^k \mu_x^{l-k}p-\mu_x^{l+k}\mu_h^2 =\frac{\mu_h\mu_x\left[\mu_h + p (\mu_x^2-1)-\mu_h r \right]}{(\mu_x-1)^2(1+\mu_x)(r-1)}
		,\\
		&&\lim_{i\to\infty}\sum_{k=0}^{I} r^k \mu_x^{i-k}p-\mu_x^{i+k}\mu_h =0,
	\end{eqnarray*}
	leading to the asymptotic variance of $Z_i$ given by \eqref{AVZi}.
\end{proof}

\section{Proof of \Cref{Theorem2}}
\label{appD}

\begin{proof}
	Define
	\begin{eqnarray*}
		a_i&:=&e^{-(\frac{1}{\tau}+\frac{\sigma^2}{2})\Delta+\sigma \xi_{i-1}}, \quad  
		\alpha_i:= \frac{e^{-(\frac{1}{\tau}+\frac{\sigma^2}{2})\Delta+\sigma \xi_{i-1}}-1}{-(\frac{1}{\tau}+\frac{\sigma^2}{2})\Delta+\sigma \xi_{i-1}}, \quad \beta_i:= 1-\sigma \rho_{i-1} +\sigma^2 \left[ \frac{3}{5} \rho_{i-1}^2 + \frac{\Delta}{30} \right].
	\end{eqnarray*}
	Due to the positivity of the exponential function, we have that $a_i>0$ for all $i=1,\ldots,N$. Since
	$(e^x-1)/x>0$ for all $x\neq 0$, 
	we have that $\alpha_i>0$ almost surely for all $i=1,\ldots,N$. Moreover, since $\sigma,\Delta>0$, the quantity $\beta_i>0$ for all $\rho_{i-1} \in \mathbb{R}$ and $i=1,\ldots,N$.
	
	We start with the piecewise linear method $\widetilde{Y}^{\textrm{Lin}}(t_i)$ defined via \eqref{Lin} and set
	\begin{equation*}
	b_i:=\widetilde{Y}^{\textrm{Lin}}(t_{i-1}), \quad c:=\mu\Delta.
	\end{equation*}
	Given $\mu \geq 0$ and $\widetilde{Y}^{\textrm{Lin}}(t_{i-1})>0$, we have that $b_i>0$ and $c\geq 0$. Thus, $\widetilde{Y}^{\textrm{Lin}}(t_i)=a_ib_i + c \alpha_i>0$ almost surely proving the \textit{discrete unattainable property}. Given $\mu=0$ and $Y_0=0$, we have that $b_1=c=0$. Thus, $\widetilde{Y}^{\textrm{Lin}}(t_1)= a_1b_1 + c \alpha_1=0$ proving the \textit{discrete absorbing property}. Given $\mu>0$ and $Y_0=0$, we have that $b_1=0$ and $c>0$. Thus, $\widetilde{Y}^{\textrm{Lin}}(t_1)= a_1 b_1 + c \alpha_1>0$ almost surely proving the \textit{discrete entrance property}. Given $\mu<0$ and $\widetilde{Y}^{\textrm{Lin}}(t_{i-1})\leq 0$, we have that $b_i \leq 0$ and $c<0$. Thus, $\widetilde{Y}^{\textrm{Lin}}(t_i)=a_i b_i + c \alpha_i<0$ almost surely proving the \textit{discrete exit property}. The properties for the log-ODE method can be proved in the same way, using the positivity of $\beta_i$.
	
	Now, we consider the Lie-Trotter splitting $\widetilde{Y}^{\textrm{L1}}(t_i)$ defined via \eqref{SP1} and set
	\begin{equation*}
	b_i:=\widetilde{Y}^{\textrm{L1}}(t_{i-1}), \quad c:=\mu\Delta.
	\end{equation*}
	Given $\mu \geq 0$ and $\widetilde{Y}^{\textrm{L1}}(t_{i-1})>0$, we have that $b_i>0$ and $c\geq 0$. Thus, $\widetilde{Y}^{\textrm{L1}}(t_i)=a_i(b_i+c)>0$ proving the \textit{discrete unattainable property}. Given $\mu=0$ and $Y_0=0$, we have that $b_1=0$ and $c=0$. Thus, $\widetilde{Y}^{\textrm{L1}}(t_1)=a_1(b_1+c)=0$ proving the \textit{discrete absorbing property}. Given $\mu>0$ and $Y_0=0$, we have that $b_1=0$ and $c>0$. Thus, $\widetilde{Y}^{\textrm{L1}}(t_1)=a_1(b_1+c)>0$ proving the \textit{discrete entrance property}. Given $\mu<0$ and $\widetilde{Y}^{\textrm{L1}}(t_{i-1})\leq0$, we have that $b_i\leq 0$ and $c<0$. Thus, $\widetilde{Y}^{\textrm{L1}}(t_i)=a_i(b_i+c)<0$ proving the \textit{discrete exit property}. The discrete boundary properties of the other three splitting methods can be proved in a similar way.
\end{proof}

\end{document}